\documentclass[oneside]{amsart}

\usepackage{todonotes}
\usepackage[T1]{fontenc}

\usepackage[latin1]{inputenc}

\usepackage{enumerate}
\usepackage{enumitem}
\usepackage{comment}

\usepackage{amsmath,amsthm,amsfonts,amssymb,latexsym, mathrsfs}

\usepackage{hyperref}

\hypersetup{
    colorlinks,
    citecolor=blue,
    filecolor=blue,
    linkcolor=blue,
    urlcolor=blue
}

\makeatletter

\theoremstyle{plain}
\newtheorem{thm}{Theorem}[section]

\numberwithin{equation}{section} 

\numberwithin{figure}{section} 

\theoremstyle{plain}
\newtheorem{cor}[thm]{Corollary} 

\theoremstyle{definition}
\newtheorem{defn}[thm]{Definition}

\theoremstyle{plain}
\newtheorem{lem}[thm]{Lemma} 

\theoremstyle{plain}
\newtheorem{prop}[thm]{Proposition} 

\theoremstyle{plain}
\newtheorem{fact}[thm]{Fact}

\theoremstyle{plain}

\theoremstyle{plain}

\newtheorem{rem}[thm]{Remark}

\newtheorem{exam}[thm]{Example}

\newtheorem{obse}[thm]{Observation}

\newtheorem{theo}[thm]{Theorem}

\newtheorem{rema}[thm]{Remark}

\newtheorem{nota}[thm]{Notation}

\makeatletter

\usepackage{amssymb}

\usepackage{amsmath}

\newcommand{\Z}{\mathbb{Z}}

\newcommand{\FF}{\mathbb{F}}
\newcommand{\cL}{\mathcal{L}}
\newcommand{\og}{\omega}
\newcommand{\cA}{\mathcal{A}}

\def\seq{\subseteq}

\newcommand{\Frac}{\mathrm{Frac}}

\newcommand{\dv}{\mathrm{div}}

\newcommand{\set}[1]{\left\{ {#1} \right\}}
\newcommand{\abs}[1]{\lvert {#1} \rvert}

\def \<{\langle}
\def \>{\rangle}

\usepackage{amsfonts}

\usepackage{amscd}

\usepackage{latexsym}

\usepackage{epsfig}

\makeatother

\def\Ind#1#2{#1\setbox0=\hbox{$#1x$}\kern\wd0\hbox to 0pt{\hss$#1\mid$\hss}
\lower.9\ht0\hbox to 0pt{\hss$#1\smile$\hss}\kern\wd0}
\def\ind{\mathop{\mathpalette\Ind{}}}
\def\Notind#1#2{#1\setbox0=\hbox{$#1x$}\kern\wd0\hbox to 0pt{\mathchardef
\nn=12854\hss$#1\nn$\kern1.4\wd0\hss}\hbox to
0pt{\hss$#1\mid$\hss}\lower.9\ht0 \hbox to
0pt{\hss$#1\smile$\hss}\kern\wd0}
\def\nind{\mathop{\mathpalette\Notind{}}}

\newcommand{\eq}{^{\operatorname{eq}}}

\newcommand{\acl}{\operatorname{acl}}
\newcommand{\spane}{\operatorname{span}}

\newcommand{\tp}{\operatorname{tp}}

\newcommand{\dcl}{\operatorname{dcl}}

\newcommand{\scl}{\operatorname{scl}}

\newcommand{\vect}[1]{\langle {#1} \rangle}

\begin{document}

\title{Vector spaces with a dense-codense generic submodule}

\date{\today}

\author{Alexander Berenstein}
\address{Universidad de los Andes,
Cra 1 No 18A-12, Bogot\'{a}, Colombia}
 \urladdr{\href{www.matematicas.uniandes.edu.co/\textasciitilde aberenst}{www.matematicas.uniandes.edu.co/\textasciitilde aberenst}}

\author[C. d'Elb\'ee]{Christian d'Elb\'ee}
\address{Mathematisches Institut der Universit\"at Bonn \\
 Germany}
\email{cdelbee@math.uni-bonn.de}
\urladdr{\href{https://choum.net/~chris/page_perso/}{https://choum.net/~chris/page\_perso/}}

\author{Evgueni Vassiliev}
\address{Grenfell Campus, Memorial University of Newfoundland, Corner Brook, NL A2H 6P9, Canada}
 \urladdr{\href{https://sites.google.com/site/yevgvas}{https://sites.google.com/site/yevgvas}}

\thanks{The authors would like to thank Nicholas Ramsey for providing the proof of Proposition \ref{nsop1uptofinite}. The first and second named authors are grateful to the ECOS Nord program 048-2019 for supporting a visit of the second named author to Universidad de los Andes, Bogot\'a, which led to this collaboration.}

\begin{abstract}
We study expansions of a vector space  $V$ over a field $\FF$, possibly with extra structure, with a generic submodule over a subring of $\FF$. We construct a natural expansion by existentially defined functions so that the expansion in the extended language satisfies quantifier elimination. We show that this expansion preserves tame model theoretic properties such as stability, NIP, NTP$_1$, NTP$_2$ and NSOP$_1$. We also study induced independence relations in the expansion.
\end{abstract}

\pagestyle{plain}

\maketitle

\tableofcontents

\section{Introduction}

This paper brings together ideas of  dense-codense expansions of geometric structures \cite{BeVa-H, BeVaLP,BHK} with ideas about generic expansions by groups by D'Elb\'ee \cite{De}. Our base structures are vector spaces over a fixed field $\FF$, possibly with extra structure, such that the algebraic closure agrees with the $\FF$-linear span. We then fix a subring $R$ of $\FF$ and we study expansions by additive $R$-submodules satisfying some form of genericity (for technical details, see the definition of $T_U$ and $T^G$ in Definition 
\ref{twobasictheories}), the main goal is to see how tame model-theoretic properties transfer from the original structure to the expansion.

There are many papers that deal with expansions by predicates and preservation of tame properties. There are general approaches \cite{CZ,ChSi} that study stable or NIP structures expanded by predicates, the main idea being that if the induced structure on the predicate is stable/NIP and if the formulas in the expansion are equivalent to bounded formulas (i.e. where the quantifiers range over the predicate), then the pair is again stable/NIP. On the other hand one can start with a geometric structure \cite{BeVa-H, BeVaLP,BHK} and study preservation of properties like NTP$_2$, strong dependence, supersimplicity or NSOP$_1$ under expansions by well-behaved dense codense predicates (for example, the predicate being an elementary substructure \cite{BeVaLP}, a collection of algebraically independent elements \cite{BeVa-H}  or a multiplicative substructure with the Mordell-Lang property \cite{BHK, vdDG, BeVaG}). This approach shares some ingredients with the previous one, formulas in the expanded language are equivalent to bounded formulas and the density property implies that the induced structure on the predicate is tame.
One can also study different generic expansions, a classical example is the generic predicate, which preserves simplicity (see \cite{ChPi}). A more general construction is due to Winkler \cite{Win} in his thesis: given a model-complete $\cL$-theory $T$, and a language $\cL' \supset \cL$, the theory $T$ can be considered as an $\cL'$-theory which has a model-companion, provided $T$ has elimination of $\exists^\infty$. Winkler \cite{Win} also considered the expansion of a theory by generic Skolem functions. Both expansions of Winkler were later shown to preserve the property NSOP$_1$ (\cite{KrRa}, \cite{Nub}). One can also consider the expansion of a theory by a predicate for a reduct of this theory, for instance expanding a theory of fields by an additive or multiplicative generic subgroup (see \cite{De, DeA, Abg}).

We start this paper (Section \ref{Sub:qefractionfield}) by studying the expansion of a theory of $\FF$-vector space (in which the algebraic closure is the vector span) by a predicate for a generic $R$-submodule, where $R$ stands for an integral domain such that $\FF=\Frac(R)=$ the fraction field of $R$. After adding predicates for $pp$-formulas in the language of $R$-modules, we characterize the expansions that are existentially closed, prove the existence of a model companion $T^G$ and in doing so show quantifier elimination for the expansion. 
As a Corollary of quantifier elimination, we show that $T^G$ is stable (resp. NIP) whenever $T$ is stable (resp. NIP).

In Section \ref{sec:basics}, we study the general case where we drop the assumption $\FF=\Frac(R)$. We construct a natural expansion by existentially defined functions so that the expansion in the extended language satisfies quantifier elimination and we prove a model-completeness result. In Section \ref{sec:NIPandstable} we use our description of definable sets to show that NIP and stability are preserved under these expansions. We would like to point out that these preservation results can be proved using the approach from \cite{BHK}.

In Section \ref{sec:preservationTP}, we prove the main results of this paper: preservation of  NTP$_1$, NTP$_2$ and NSOP$_1$ in the expansion. The approach we follow is to check the property by doing a formula-by-formula analysis, separating the cases when the corresponding definable set is \emph{small} (algebraic over the predicate) or \emph{large}. To show the preservation of NSOP$_1$ we build on ideas presented in \cite{Ra}, the proofs for the preservation of NTP$_1$, NTP$_2$ generalize ideas presented in \cite{BeKi,DoKi}. In Section \ref{sec:G-basis}, we study independence notions in the expansion, assuming the original theory has a good notion of independence. As a Corollary, we give another proof for the preservation of simplicity assuming $T$ has $SU$-rank one
and $\FF=\Frac(R)$. In Section \ref{sec:moreexamples}, we introduce a candidate for an example of an non-simple NSOP$_1$ pregeometric theory with modular pregeometry.

\section{A first example: the case $\FF = \Frac(R)$}\label{Sub:qefractionfield}

Let $\FF$ be a field and let $R$ be a subring of $\FF$ such that $\FF = \Frac(R)$ (the fraction field of $R$). Let $\cL_0 = \set{(\lambda\cdot)_{\lambda\in \FF}, +,0}$ be the language of vector spaces over $\FF$ and let $\cL = \set{+,0,\set{\lambda \cdot}_{\lambda\in \FF},\dots}$ be an extension. Let $T$ be a complete $\cL$-theory that expands the theory of vector spaces over $\FF$ which has quantifier elimination in $\cL$ and that satisfies the following properties:\\
(i) Whenever $M\models T$ and $\vec a\in M^n$, $\dcl(\vec a) = \acl(\vec a) = \spane_{\FF}(\vec a)$.\\ (ii) It eliminates the quantifier $\exists^\infty$.

Let $G$ be a unary predicate and for each formula $\phi(\vec x)$ in the language $\cL_{R-mod}=\set{+,0,(r\cdot)_{r\in R}}$ of $R$-modules, let $P_{\phi}(\vec x)$ be a new predicate. Let $\cL_G$ be the expansion of $\cL$ by $G$ and $P_{\phi}$ for all formulas $\phi$ in $\cL_{R-mod}$.
We will consider pairs $(V,G)$ in the language $\cL_G$ that satisfy the following first order conditions:
\begin{enumerate}[label=(\Alph*)]
    \item $G$ is an $R$-module and for all formulas
    $\phi(\vec x)$ in the language of $R$-modules, $\forall \vec x (P_\phi(\vec x) \to G(\vec x))$ and 
    $\forall \vec x (G(\vec x) \to (P_\phi(\vec x) \leftrightarrow \hat \phi(\vec x))$, where $\hat \phi(\vec x)$ in the relativization of $\phi(\vec x)$ to $G$ and we interpret the relativization seeing $G$ as an $R$-module.
    \item (For all $r\in R\setminus\{0\}$, $r G$ is dense in $V$). For every $\mathcal{L}$-formula $\phi(x,\vec y)$, the axiom $\exists^\infty x \phi(x,\vec y)\rightarrow \exists x (\phi(x,\vec y)\wedge r G(x))$;
    \item (Extension/co-density property) for any $\cL$-formulas $\phi(x,\vec y)$ and $\psi(x,\vec y,\vec z)$ and $n\ge 1$, the axiom 
    $$(\exists^\infty x\phi(x,\vec y)\wedge \forall\vec z \exists^{\le n}x\psi(x,\vec y,\vec z))\to \exists x (\phi(x,\vec y)\wedge \forall \vec z (G(\vec z)\to \neg\psi(x,\vec y,\vec z))).$$
\end{enumerate}

Let $T^G$ be the $\cL_G$ theory satisfying the schemes $(A)$, $(B)$ and $(C)$. In all models $(V,G)$ under consideration, the predicate $G$ will interpret an $R$-module. Since modules have quantifier elimination up to boolean combinations of pp-formulas (see \cite{Zie}), we can always assume that the predicates $P_{\phi}(\vec x)$ are only defined for positive and negative
instances of pp-formulas. Throughout this paper we will assume the reader is familiar with basic properties of pp-formulas inside $R$-modules.

\begin{nota} In what follows, whenever $(V,G)\models T^G$ and $A\subseteq V$, we will write $G(A)$ for $A\cap G(V)$.
\end{nota}

Our first goal is to show that the theory $T^G$ has quantifier elimination. We will start by proving properties of the divisible elements and the $\mathcal{L}$-terms.

\begin{lem}[Density of $G^\dv$]\label{lm_densediv}
Let $(V,G)$ be an $\abs{\FF}^+$-saturated model of $T^{G}$. Let $B\subseteq V$ such that $\abs{B}\leq \abs{\FF}$ and $p(x)$ be a consistent non-algebraic $\mathcal{L}$-type in a single variable over $B$. Let $G^\dv = \bigcap_{r\in R\setminus \{0\}} r G$. Then $p(V)\cap G^\dv(V)$ is infinite.
\end{lem}

\begin{proof}
By compactness, it is sufficient to show that for all $r_1,\dots,r_n\in R$, $p(x)$ has infinitely realisations in $\bigcap_{i=1}^{n} r_iG$. Let $r = r_1\dots r_n$. Now note that by axiom (B), the type $rG(x)\wedge p(x)$ has infinitely many realisations.
\end{proof}

\begin{lem}\label{disj_terms}
  Let $t(\vec x)$ be an $\cL$-term, then there exists $\cL$-formulas $\theta_1(\vec x),\dots,\theta_n(\vec x)$ forming a partition of the universe, (i.e. such that $V\models \forall \vec x ( \bigvee_i \theta_i(\vec x)) \wedge \bigwedge_{i\neq j} \neg \exists \vec x \theta_i(\vec x)\wedge \theta_j(\vec x)$), and $\cL_0$-terms $t_1(\vec x), \dots, t_n(\vec x)$ such that
  \[t(\vec x) = y \leftrightarrow \bigvee_i t_i(\vec x)=y\wedge \theta_i(\vec x).\]
\end{lem}
\begin{proof}
  Since $T$ satisfies property $(i)$, the algebraic closure agrees with the $\FF$-span, hence $\set{t(\vec x) = y}\cup \set{\vec\lambda \cdot \vec x \neq y\mid \vec \lambda\in \FF^n}$ is inconsistent. Thus we can find $\vec \lambda^1, \dots, \vec \lambda^n\in \FF^n$ such that $t(\vec x) = y \rightarrow \bigvee_i \vec \lambda^i \cdot \vec x = y$. We may assume that $\vec \lambda^i \vec x=y$ define disjoint vector spaces. Now choose $t_i(\vec x) = \vec \lambda^i \cdot \vec x$ and let $\theta_i(\vec x)$ be the formula $t_i(\vec x) = t(\vec x)$. 
\end{proof}

\begin{thm}\label{theo_QEfrac}
 The theory $T^G$ has quantifier elimination.
\end{thm}
\begin{proof}
We show that the set of partial isomorphisms between two $\abs{\mathcal{L}}^+$-saturated models $(V,G), (V', G')$ of $T^{G}$ has the back and forth property. Let $B\subset V, B'\subset V'$ be two small substructures (i.e. $B = \vect B$ and $B' = \vect{B'}$, $|B|,|B'|<\abs{ \mathcal{L}}^+$) such that there exists a partial isomorphism $\sigma: B \rightarrow B'$.

Let $a\in V\setminus B$. Since $B=\spane_\FF(B)$ and $T$ satisfies property $(i)$, we have that $\tp_\cL(a/B)$ is non-algebraic. Every formula in $\tp^{QF}_{\cL_G}(a/B)$, the quantifier free type in the extended language, is equivalent to a disjunction of formulas of the form 
 \begin{align*}
     &\phi(x,\vec b)\wedge P_\psi(t_1(x,\vec b)),\dots, t_k(x,\vec b))\wedge  \neg P_{\psi'}(t_1'(x,\vec b)),\dots, t_k'(x,\vec b))\\
     &\wedge \bigwedge_i t_i(x,\vec b)\in G\wedge \bigwedge_j t_j(x,\vec b)\notin G\wedge \bigwedge_k t_k(x,\vec b)=0\wedge \bigwedge_l t_l(x,\vec b)\neq 0.
 \end{align*}
 where $\phi(x,\vec y)$ is an $\cL$-formula and each of $t_i(x,\vec b)$, $t_j(x,\vec b)$,
 $t_k(x,\vec b)$, $t_l(x,\vec b)$ is an $\cL$-term.
 By Lemma \ref{disj_terms}, up to a finite disjunction, we may assume that each term is an $\cL_0$-term, i.e. of the form $\lambda x+ b$ for $b\in \spane_\FF(B)$ and $\lambda\in \FF\setminus \{0\}$. 
 As $\FF = \hat R$, conditions of the form 
 $P_\psi(t_1(x,\vec b),\dots, t_k(x,\vec b))\wedge  \neg P_{\psi'}(t_1'(x,\vec b),\dots, t_k'(x,\vec b))$ with $\cL_0$-terms are equivalent to a single $P_\theta(\vec t(x,\vec b))$ (for $\vec t(x,b) = (t_1(x,\vec b),\dots,t_k(x,\vec b))$ a tuple of linear combinations of $x$ and $\vec b$ with coefficients in $R$) up to a finite disjunction of formulas of the form $t_j(x,\vec b)\notin G$. As $a\notin B$, conditions of the form $t_k(x, \vec b) = 0$ do not appear in $\tp^{QF}_{\cL_G}(a/B)$. It turns out we only need to consider formulas of the form
\[\phi(x,\vec b)\wedge P_\psi(\vec t(x,\vec b))
     \wedge \bigwedge_i t_i(x,\vec b)\in G\wedge \bigwedge_j t_j(x,\vec b)\notin G\wedge \bigwedge_l t_l(x,\vec b)\neq 0.\]

 We will extend the map $\sigma$ by cases.
 
 \underline{Step 1.} If $a\in G$. Then, as $\FF=\hat R$, formulas of the form $qx+b\in G$, where $q\in \FF$, are equivalent to formulas of the form $rx+g\in r'G$ for some $r,r'\in R$ and $g\in G(B)$ so conditions of the form $rx+g\in r'G\wedge x\in G$ are equivalent to $P_\psi (x,g)$ for some $\cL_{R-mod}$-formula $\psi$. By quantifier elimination in $R$-modules \cite{Zie}, the condition $P_\psi(\vec t(x, \vec b))$ is equivalent to a boolean combination of formulas of the form $rx+\vec r \cdot \vec b \in s_1G+\dots + s_n G$. If $x\in G$, the latter is equivalent to $rx+ g\in s_1G+\dots + s_n G$ for some $g\in G(B)$, so $P_\psi(\vec t(x, \vec b))$ is equivalent to some $P_{\psi'}(x, \vec g)$  and tuple $\vec g\in G(B)$. Let $q(x) = \set{P_\psi(x,\vec g)\mid G\models \psi(a,\vec g), \psi\in\cL_{R-mod}, \vec g\in G(B)}$, $q^\sigma(x) = \sigma(q)(x)$ and $p(x) = \sigma(\tp_\cL(a/B))$. To extend $\sigma$ it is enough to show that there are infinitely many realisations of $p(x)\cup q^\sigma(x)$ in $V'$. First using the fact that $G(V')$ is $\abs{\FF}^+$-saturated and that $|B'|<|\FF|^+$, there is $a''\models q^\sigma(x)$. Let $p_{sh}(x) = p(x+a'')$, the type shifted by $a''$. The type $p_{sh}$ is also non-algebraic, hence by density of $R$-divisible elements (Lemma \ref{lm_densediv}), there exists infinitely many $d\in G^{div}$ such that $d\models p_{sh}(x)$. For any such $d$, let $a'=d+a''$.

Claim: $a'\models q^\sigma(x)\cup p(x)$. First, as $d\models p_{sh}(x)$, $a' = d+a''\models p(x)$. By quantifier elimination in $R$-modules \cite{Zie}, every formula in $q^\sigma$ is a boolean combination of conditions of the form $ra'+g \in r_1G+\dots+r_nG$ for $r \in R$ and $g\in G(B)$. As $d\in G^\dv$, we have that $rd\in r_1G$ for all $r\in R$, hence $ra'+g \in r_1G+\dots+r_nG$ if and only if $ra+g\in r_1G+\dots+r_nG$, so since $a''\models q^\sigma(x)$, we also have $a'\models q^\sigma(x)$. It follows that $q^\sigma(x)\cup p(x)$ has infinitely many realisations.\\
\underline{Step 2.} If $a\in \spane_\FF(G(V)B)$. Then $a = q_1g_1+\dots +q_ng_n + b$ for some $\vec q\in \FF$ and $\vec g\in G$. Applying step 1 to the elements $g_1,\dots,g_n$, we can extend $\sigma$ to a partial isomorphism (which we still call $\sigma$) that includes $g_1,\dots,g_n$ in its domain. Now set $\sigma(a) = q_1\sigma(g_1)+\dots+q_n\sigma(g_n)+\sigma(b)$.\\
\underline{Step 3.} If $a\notin \spane_\FF(GB)$ then formulas of the form $qx+b\in G$ do not appear in $\tp^{QF}_{\cL_G}(a/B)$, and the formula $P_\psi(x,\vec g)$ belongs to $\tp^{QF}_{\cL_G}(a/B)$ only when $\psi$ is the negation of a pp-formula. Thus it is enough to show that $\sigma(\tp_\cL(a/B))$ have infinitely many realisations in $V'\setminus \spane_\FF(GB')$, which follows easily from condition $(C)$, compactness, the fact that $|B|'<|\FF|$ and the fact that $(V',G')$ is $\abs{\FF}^+$-saturated.
\end{proof}

\begin{cor}\label{back-and-forthQE}
Let $(V,G)$ and $(V',G')$ be two models of $T^G$. Then whenever $\vec a\in V$, $\vec a'\in V'$ are two tuples of the same length such that 
\begin{enumerate}
\item $\tp_{R-mod}(G(\spane_{R}(\vec a)))= \tp_{R-mod}(G(\spane_{R}(\vec a')))$ (the types 
agree in the sense of R-modules\footnote{By which we mean that 
\[\set{P_{\psi}(\vec x)\mid G\models \psi(\vec g), \vec g\in G(\spane_\FF(\vec a)),\psi\in \cL_{R-mod}} = \set{P_{\psi}(\vec x)\mid G\models \psi(\vec g'), \vec g'\in G(\spane_\FF(\vec a')), \psi\in \cL_{R-mod}} .\]});
\item $\tp_{\mathcal{L}}(\vec a)=\tp_{\mathcal{L}}(\vec a')$ (their types agree in the language $\mathcal{L}$) ;
\end{enumerate}
Then $\tp_G(\vec a)=\tp_G(\vec a')$.
\end{cor}
\begin{proof}
From Theorem \ref{theo_QEfrac}, $T^G$ has quantifier elimination, hence every formula in $\tp_G(\vec a)$ is equivalent to a disjunction of formulas of the form
\[\phi(\vec a)\wedge P_\psi(\vec t(\vec a))\wedge  \bigwedge_i t_i(\vec a)\in G\wedge \bigwedge_j t_j(\vec a)\notin G\]
for some quantifier-free $\cL$-formula $\phi(\vec x)$, $\psi$ an $R$-module formula and $\cL$-terms $(t_i(\vec x),t_j(\vec x))_{i,j}$, $\vec t(\vec x)$. Using Lemma~\ref{disj_terms}, we may assume that terms are $\FF$-linear combinations and that $\vec t(\vec a)$ is a tuple of $R$-linear combinations (as in the proof of Theorem \ref{theo_QEfrac}) in $G$.
It follows that $\tp_G(\vec a)$ is equivalent to a set of formulas of the form \[\phi(\vec a)\wedge P_\psi(\vec t(\vec a))\wedge \bigwedge_i \vec q_i\cdot \vec a\in G\wedge \bigwedge_j \vec q_j\cdot \vec a\notin G\]
for $\vec q_i,\vec q_j\in \FF$. As $\FF = \Frac(R)$, each condition of the form $\vec q_i\cdot \vec a\in G$ is equivalent to a condition $\vec r_i\cdot \vec a\in s_iG$ for $\vec r_i,s_i\in R$. From condition \textit{(1)}, $\vec r_i\cdot \vec a\in s_iG$ if and only if $\vec r_i\cdot \vec a'\in s_iG$.
Also by \textit{(1)}, we have that $P_\psi(\vec t(\vec a))$ holds if and only if $P_\psi(\vec t(\vec a'))$ holds. Finally by condition (2), for $\phi(\vec x)$ an $\mathcal{L}$-formula, $\phi(\vec a)$ holds if and only if $\phi(\vec a')$ holds.  This proves the desired result.
\end{proof}

\begin{cor}[to the proof of Theorem \ref{theo_QEfrac}]
Let $(V,G)$ be a model of $T^{G}$ and let $B\subseteq V$. Then 
\[\acl_G(B) = \dcl_G(B) = \spane_\FF(B).\]
\end{cor}

\begin{proof}
    We follow the proof of Theorem \ref{theo_QEfrac}. In case 1, ($a\in G$ and $a\notin \vect{B}$), observe that $a$ is not in $\acl_{R-mod}(G(B))$ hence $q=\tp_{R-mod}(a/G(B))$ is a non-algebraic type so neither is $q^\sigma$, thus $q^\sigma\cup \tp_{\cL}(a/B)$ has infinitely many realisations. The other two cases are similar.  
\end{proof}

\begin{lem}\label{lm_pregeom_SAP}
Let $T$ be a complete pregeometric theory with quantifier elimination. Then $T$ has SAP. 
\end{lem}

\begin{proof}
We first prove the following claim.\\
\noindent \textbf{Claim}. Let $T$ be a complete theory with quantifier elimination. Then $T$ has AP. 

Let $M_0,M_1,M_2\models T$ and assume there are embeddings $f_i:M_0\rightarrow M_i$ for $i=1,2$. By quantifier elimination the maps are elementary embeddings. Let $\kappa$ be the biggest cardinal among $|M_1|, |M_2|$. Let $N\models T$ be $\kappa^+$-saturated and $\kappa^+$-strongly homogeneous. Then 
by $\kappa^+$-saturation for each $i=1,2$ there is an elementary map $g_i:M_i\to N$. Since $g_1(f_1(M_0))\equiv g_2(f_2(M_0))$ by strong homogeneity there is an automorphism $h$ of $N$ such that
for each $m_0\in M_0$, $h(g_2(f_2(m_0)))=g_1(f_1(m_0))$. Then
$h(g_2(M_2))$ is an elementary copy of $M_2$, $g_1(M_1)$ an elementary copy of $M_1$
and for each $m_0\in M_0$ we have $h(g_2(f_2(m_0)))=g_1(f_1(m_0))$ and the AP holds.

We now prove SAP. Let $M_0,M_1,M_2$ be three models of $T$ and $f_i:M_0\rightarrow M_i$ be embeddings. Since $T$ has the AP there exists a model $M_3$ of $T$ and $g_i:M_i\rightarrow M_3$ such that $g_1\circ f_1=g_2\circ f_2$. Let $M_0' = g_1\circ f_1(M_0)$, $M_1' = g_1(M_1)$, $M_2' = g_2(M_2)$. The type of $M_1'$ over $M_0'$ has a free extension with respect to $M_2'$ (in the sense of the pregeometry), hence in a monster model $N$ of $T$ containing $M_3$ there exists $M_1'' \equiv_{M_0'} M_1'$ such that $M_1''\subset N$ and $M_2'$ are independent over $M_0'$, in particular, as $M_0'$, $M_1''$ and $M_2'$ are algebraically closed, $M_1''\cap M_2' = M_0'$. Let $\sigma$ be an automorphism of $N$ over $M_0'$ and such that $\sigma(M_1') = M_1''$. Then $g_1' = (\sigma\upharpoonright M_1')\circ g_1$ is an embedding $M_1\rightarrow N$ and $g_1'(M_1)\cap g_2(M_2) = M_1''\cap M_2' = M_0' = g_1'\circ f_1 (M_0)$.
\end{proof}

\begin{rem}
The previous lemma actually holds without the assumption of $T$ being pregeometric. One could use for instance \cite[Proposition 1.5]{Adler09} to get the existence of such $M_1''$.
\end{rem}

\begin{prop}\label{TU-modelcompletion}
Let $T_U$ be the $\cL_G$-theory satisfying condition $(A)$. If $T_U$ is inductive, then $T^G$ is the model-completion of $T_U$.
\end{prop}

\begin{proof}
We first check that every model of $T_U$ extends to a model of $T^G$. This is done by a standard chain argument. For example, if $(V,G)\models T_U$, $\phi(x,\vec y)$ is an $\mathcal{L}$-formula, $r\in R\setminus \{0\}$ and $\vec a\in V$ is such that $V\models \exists^{\infty}x \phi(x,\vec a)$, we let $V_1$ be a proper extension of $V$ such that $\phi(V_1,\vec a)\setminus \phi(V,\vec a)$ is infinite. Then one chooses $b_1\in \phi(V_1,\vec a)\setminus \phi(V,\vec a)$ and defines $G_1=\langle G,\frac{b_1}{r}\rangle $ the smallest $R$-module containing $G$ and $\frac{b_1}{r}$. One may apply this argument for all tuples $\vec a\in V$ with $V\models \exists^{\infty}x \phi(x,\vec a)$ and assume that $\phi(rG_1,\vec a)\neq \emptyset$. Then we repeat the process for tuples in $V_1$ and build a chain $\{(V_i,G_i)\}_{i=1}^\infty$ whose union has the desired properties.

Second, one needs to check that $T_U$ has the amalgamation property. In fact, $T_U$ has the strong amalgamation property. Let $(V_0,U_0)$, $ (V_1,U_1)$ and $(V_2,U_2)$ be three models of $T_U$ such that there exists embeddings $f_i:(V_0,U_0)\rightarrow (V_i,U_i)$ for $i=1,2$. By Lemma~\ref{lm_pregeom_SAP}, $T$ has SAP hence there exists $V_3\models T$ and $\cL$-embeddings $g_i : V_i\rightarrow V_3$ ($i=1,2$) such that $g_1\circ f_1=g_2\circ f_2$ and $g_1(V_1)\cap g_2(V_2) = g_1\circ f_1(V_0) = g_2\circ f_2(V_0)$. Define $U_3 = g_1(U_1)+g_2(U_2)$, then $(V_3,U_3)$ is a model of $T_U$. We have to check that $g_i:(V_i,U_i)\rightarrow (V_3,U_3)$ are $\cL_U$-embeddings, i.e. $g_i(V_i)\cap U_3 = g_i(U_i)$ for $i=1,2$. Without loss of generality, we may assume that $i=1$. We have that $g_1(V_1)\cap (g_1(U_1)+g_2(U_2)) = g_1(U_1)+g_1(V_1)\cap g_2(U_2)$. Also $g_1(V_1)\cap g_2(V_2)=g_2\circ f_2(V_0)$ hence $g_1(V_1)\cap g_2(U_2)\subseteq g_2\circ f_2(V_0)$. Now as $g_2$ is injective and $f_2$ is an $\cL_U$ embedding, we have that $g_2(U_2)\cap g_2\circ f_2(V_0)\subseteq g_2\circ f_2(U_0)$ hence $g_1(V_1)\cap g_2(U_2)\subseteq g_2\circ f_2(U_0)= g_1\circ f_1(U_0)\subseteq g_1(U_1)$. It follows that $g_1(V_1)\cap (g_1(U_1)+g_2(U_2)) = g_1(U_1)$, which finishes the argument. 

Finally note that $T^G$ has quantifier elimination and thus it is model complete.
\end{proof}

It follows from Theorem \ref{theo_QEfrac} that the theory $T^G$ has quantifier elimination in the extended language. We will use this fact below to show that several tameness properties are preserved in the expansion. 

We start with a small lemma showing the stability of the new predicate.

\begin{lem}\label{Gisstablemodelcompanion}
Let $t(\vec x,\vec y)$ be a $\mathcal{L}_0$-term, then the formula $G(t(\vec x,\vec y))$ is stable.
\end{lem}

\begin{proof}
Assume that $t(\vec x,\vec y) = \vec \alpha \cdot \vec x+\vec \beta \cdot \vec y$ and that there exist sequences $(\vec a_i)_{i\in \omega}$ and $(\vec b_i)_{i\in \omega}$ such that 
$t(\vec a_i,\vec b_j)\in G$ if and only if $i<j$. 
Now let $c_i=\vec \alpha \cdot \vec a_i$ and $d_i=-\vec \beta \cdot \vec b_i$. Then we have
$c_i-d_j\in G$ if and only if $i<j$. In particular we get $c_1\in d_2+G$, $c_1\in d_3+G$ and
$c_2\in d_3+G$. So the cosets $d_3+G$, $d_2+G$ are equal
and thus $c_2\in d_2+G$, which implies that 
$c_2-d_2\in G$, a contradiction.
\end{proof}

\begin{cor}\label{preservationresultsmodelcompanion}
If $T$ is stable, then $T^G$ is stable. If $T$ is NIP, then so is $T^G$.
\end{cor}

\begin{proof}
By Theorem \ref{theo_QEfrac} every formula in the extended language $\mathcal{L}_G$ is equivalent to a disjunction of conjunctions of $\mathcal{L}$-formulas, $R$-module formulas and formulas of the form $t(x,\vec y)\in G$ and their negation. Since the theory of modules over any ring is stable, any $R$-module formula is stable. 

Assume now that $T$ is stable, then any $\mathcal{L}$-formula is stable. By Lemma \ref{disj_terms}, any formula of the form $t(x,\vec y)\in G$, where $t(\vec x,\vec y)$ be a $\mathcal{L}$-term, is equivalent to a formula of the form $\bigvee_i t_i(x, \vec y)\in G \wedge \theta_i(x, \vec y)$ where $\theta_i(x, \vec y)$ is an $\mathcal{L}$-formula and each $t_i(x, \vec y)$ is a $\mathcal{L}_0$-term. By Lemma \ref{Gisstablemodelcompanion} the formulas $t_i(x, \vec y)\in G$ are stable for $i\leq n$ and by hypothesis $\theta_i(x, \vec y)$ is also stable, thus $t(x,\vec y)\in G$ is stable. Thus if $T$ is stable so is $T^G$. The same argument works for preservation of $NIP$.
\end{proof}

We even get some easy properties about strong dependence and the stability spectrum based on properties of the subgroups of $G$:

\begin{cor}\label{stableandnotsuperstable}
Assume that $char(\FF)=0$ and that for infinitely many primes $\{p_i: i\in I\}$ we have that the index of $p_iG$ in $G$ is infinite. Then if $T$ is stable we have that $T^G$ is strictly stable. If $T$ is NIP then $T^G$ is not strongly dependent.
\end{cor} 

\begin{proof}
One can build an array using the collection of groups $\{p_iG: \ \ i\in I \}$ and their cosets (see for example \cite{BeVaG}) to show that the expansion does not preserve strong dependence.
Similarly, if we order the primes in the list as $p_1,p_2,\dots$ then
the chain of definable groups $p_1G, (p_1p_2)G,\dots$ is strictly descending and the expansion can not be superstable. 
\end{proof}

We end this section with some examples and comments relating our work with the general perspective from \cite{BHK}.

\begin{exam}\label{purevectorspace}
Let $T$ is the theory of a pure vector space over a field $\mathbb{F}$ of characteristic zero. This is a strongly minimal theory and thus geometric and it satisfies that $\acl=\spane_{\FF}$. 
As we saw before the corresponding theory $T^G$ is stable and we will see below how the stability spectrum of the expansion will vary according to the choice of $R$.

Assume first that $R=\mathbb{F}=\mathbb{Q}$, then the corresponding theory $T^G$ is the theory of beautiful pairs and it will be $\omega$-stable of Morley rank equal to $2$ (see for example \cite{Bue} for the value of Morley rank in beautiful pairs of strongly minimal theories). It is the model companion of the theory of pairs of models of $T$ (see for example \cite{Va}).

On the opposite end, consider the case $R=\mathbb{Z}$. By Corollary \ref{stableandnotsuperstable} the corresponding expansion $T^G$ will be stable, not superstable. It is the model companion of the theory of $T$ expanded by a subgroup. Note that this is vector space-like  phenomenon, the theory of a field of characteristic $0$ with a predicate for an additive subgroup does not have a model-companion \cite{De}. 

Assume now that $R = \Z_{2\Z}$, the rationals whose denominator is relatively prime with $2$. Then
the collection of definable groups $\{2^nG: n\in \omega\}$ is a descending chain and $T^G$ is not superstable. \textbf{Question.} \textquestiondown Is $T^G$ dp-minimal?
\end{exam}

\begin{prop}\label{abelianstructure} Let $(V,G)\models T^G$ and assume that $V$, seen as an $\cL$-structure, is an abelian structure. Then $(V,G)$ is an abelian structure.
\end{prop}

\begin{proof}
We show that all definable subsets in the pair are again boolean combinations of cosets of $\emptyset$-definable groups. By Theorem \ref{theo_QEfrac} the expansion has quantifier elimination
and it suffices to show that atomic formulas in the pair $(V,G)$ give rise to definable sets that also have this property. Since $\mathcal{L}$-definable sets have the desired property, we only need to consider definable sets given by $G(t(\vec x, \vec a))$ where $t(\vec x,\vec y)$ is an $\mathcal{L}$-term and for a predicates of the form $P_{\phi}(x,\vec b)$. Consider first the case of a formula of the form $G(t(\vec x, \vec a))$. Assume first that there are $\lambda_1,\dots,\lambda_n,\mu_1,\dots,\mu_k\in \FF$ such that $t(\vec x,\vec y)=\lambda_1 x_1+\dots +\lambda_n x_n+\mu_1y_1+\dots+\mu_ky_k$. We now consider two cases.\\
Case 1: Assume that $\lambda_1,\dots,\lambda_n=0$. Then $G(t(\vec x, \vec a))$ holds if and only if $\mu_1a_1+\dots+\mu_ka_k\in G$. If $\mu_1a_1+\dots+\mu_ka_k\in G$ then $G(t(V^n, \vec a))=V^n$ and if $\mu_1a_1+\dots+\mu_ka_k\not \in G$ then $G(t(V^n, \vec a))=\emptyset$ and both sets are boolean combinations of $\emptyset$-definable groups.\\
Case 2: Assume that for some $i$, we have $\lambda_i\neq 0$. Choose $\vec b=b_1,\dots,b_n\in V$ such that $\lambda_1 b_1+\dots +\lambda_n b_n=\mu_1a_1+\dots+\mu_k a_k$. Then $G(t(\vec x, \vec a))$ holds if and only if $\lambda_1(x_1-b_1)+\dots+\lambda_k(x_k-b_k)\in G$. Now observe that $H=\{\vec x\in V^n:\lambda_1x_1+\dots+\lambda_nx_n\in G\}$ is a $\emptyset$-definable group and 
the coset $H-\vec b$ agrees with $G(t(V^n, \vec a))$. 

Assume now that the term $t(\vec x)$
agrees with $t_i(\vec x)$ for some $\cL_0$-terms $\{t_i(\vec x): i\leq n\}$ and some partition $\{\theta_i(\vec x): i\leq n\}$ given by $\cL$-formulas. Then
$G(t(\vec x))$ holds if and only if 
$G(t_i(\vec x))$ holds when $\theta_i(\vec x)$ holds. Since both $\theta(\vec x)$ and $G((t_i(\vec x))$ are boolean combinations of cosets, so is their conjuntion as well as disjuntions of families of this form. 

On the other hand, all definable sets in a module are boolean combination of cosets of $\emptyset$-definable groups, in particular this will be the case for $P_{\phi}(x,\vec b)$.
\end{proof}

\begin{cor}\label{1basedgroupscase1} Let $T$ be the theory of a pure vector space over a field $\mathbb{F}$ and let $R$ be a subring.
Then $T^G$ is $1$-based.
\end{cor}

\begin{proof}
Since $V$ is a $1$-based group, it is an abelian structure. Now apply the previous Proposition.
\end{proof}

\begin{exam}\label{orderedvectorspace}
In this example we deal with the base structure $\mathcal{V}=(V, +,0,<,\set{\lambda_r}_{r\in \mathbb{F}})$ and we assume that $\mathbb{F}$ is an ordered field and $V$ is an ordered vector space over $\mathbb{F}$. The theory of 
$\mathcal{V}$ is dense o-minimal, has quantifier elimination and $\acl=\spane_{\FF}$. In particular it is geometric.

We first consider the case where $R=\mathbb{Z}$ and
$\FF=\Frac(\mathbb{Z})=\mathbb{Q}$. Since $T$ has NIP, so does $T^G$. Moreover since $nG$ has infinite index in $G$ for $n>1$ the expansion is not strongly dependent.

Now consider the case where $\FF=R=\mathbb{Q}$, then the theory $T^G$ agrees with the theory of lovely pairs (dense pairs) and it is strongly dependent of $dp$-rank two (see for example \cite[Lemma 3.5]{DoGo}).
\end{exam}

Finally, we point a few connections with the setting introduced in \cite{BHK} that can also be used to analize this family of expansions. Following the terminology from \cite{BHK}, we can define the languages $\mathcal{L}_\beta=\mathcal{L}=\{+,0,\{\lambda\}_{\lambda \in \FF},\dots\}$ and $\mathcal{L}_\alpha=\{+,0,\{\lambda\}_{\lambda\in R}\}$ and define  $T_\beta=T=Th(V,+,0,\{\lambda\}_{\lambda\in \FF},\dots)$ and $T_\alpha$ to be the theory of $R$-modules. The theory $T_\beta$ is geometric and by the results we proved in this section $T^G$ is an example of a \emph{Mordell-Lang theory of pairs} in the sense of Definition 2.6 \cite{BHK}. In particular, by \cite[ Corollary 3.6]{BHK} the theory $T^G$ is near model complete (we prove that it has quantifier elimination, furthermore we show in Proposition \ref{TU-modelcompletion} that it is the model-completion of $T_U$ when $T^U$ is inductive) and Corollary \ref{preservationresultsmodelcompanion} follows from \cite[section 4]{BHK}. There are some other properties of the pair that are studied \cite{BHK} that are not addressed in this paper; among others, it follows by \cite[Theorem 4.5]{BHK} that the family considered in Example \ref{orderedvectorspace} has o-minimal open core.

\section{V-structures: Back-and-forth  and first properties}\label{sec:basics}

We now turn to the general case, we do not assume anymore that $\FF = Frac(R)$, instead we suppose that $\FF$ is a field of characteristic zero and that $R$ is a subring of $\mathbb{F}$. Let $\cL_0 = \set{+,0,\set{\lambda \cdot}_{\lambda\in \FF}}$
and let $\cL\supset \cL_0$ be an extension. Let $T$ be an $\cL$-theory expanding the theory of vector spaces over $\FF$ which has quantifier elimination in $\cL$ for which $\dcl = \acl = \spane_{\FF}$ and such that it eliminates the quantifier $\exists^\infty$. 

\begin{nota}
Throughout the rest of the paper, we will denote $\hat R = Frac(R)$.
\end{nota}

Let $G$ be a unary predicate and for each formula $\phi(\vec x)$ in the language $\cL_{R-mod}=\{ +,0,(r\cdot)_{r\in R} \}$ of $R$-modules, let $P_{\phi}(\vec x)$ be a new predicate. Let $\cL_G$ be the expansion of $\cL$ by $G$ and $P_{\phi}$ for all formula $\phi$ in $\cL_{R-mod}$. We will consider pairs $(V,G)$ in the language $\cL_{G}$ such that $V\models T$ and that also satisfy the following first order conditions
\begin{enumerate}[label=(\Alph*)]
    \item $G$ is a proper $R$-submodule of the universe, and for all $\vec a\in V$, $P_\phi(\vec a)$ if and only if $\vec a\in G$ and $G\models \phi(\vec a)$ as an $R$-module.
    \item If $\lambda_1,\dots,\lambda_n\in \FF$ are $\hat R$-linearly independent, then for all $g_1,\dots,g_n\in G$
    \[\lambda_1g_1+\dots+\lambda_n g_n = 0 \implies \bigwedge_i g_i = 0.\]
    \item (Density Property) for all $r\in R\setminus\{0\}$, $r G$ is dense in the universe. This is a first order property that can be axiomatized through the scheme: for every $\mathcal{L}$-formula $\phi(x,\vec y)$, add the sentence $\exists^\infty x \phi(x,\vec y)\rightarrow \exists x (\phi(x,\vec y)\wedge r G(x)$);
    \item (Extension/co-density property) for any $\cL$-formulas $\phi(x,\vec y)$ and $\psi(x,\vec y,\vec z)$ and $n\ge 1$, the following sentence
    $$(\exists^\infty x\phi(x,\vec y)\wedge \forall\vec z \exists^{\le n}x\psi(x,\vec y,\vec z))\to \exists x (\phi(x,\vec y)\wedge \forall \vec z (G(\vec z)\to \neg\psi(x,\vec y,\vec z))).$$
\end{enumerate}

\begin{rema}
Since our ambient structure is a vector space $V$ over $\FF$ satisfying $\dcl = \acl = \spane_{\FF}$, we can be a little more explicit in axiom scheme $(D)$. Instead of listing all $\psi(x,\vec y,\vec z)$ such  that $\exists^{\le n}x\psi(x,\vec y,\vec z))$ holds, we could simply list all finite disjuntions of linear equations with coefficients in $\FF$ in the variables $x,\vec y,\vec z$ with nontrivial coefficient in $x$.
\end{rema}

\begin{defn}\label{twobasictheories}
Let $T_U$ be the theory consisting of $T$ together with the schemes $(A)$ and $(B)$ and $T^G$ the theory consisting of adding $(A),(B),(C),(D)$.
\end{defn}

 In order to understand definable sets in $T^G$, we will consider an expansion by definition of $T_U$ and $T^G$, see Definition~\ref{df_Tplus}.

\begin{rem}
Assume that $\vec \lambda\in \FF^n$ is $\hat R$-independent, and that in a model of $T_U$ we have $a = \lambda_1 g_1+\dots+\lambda_n g_n$ for some $g_i\in G$. Then, axiom (B) implies that such collection $(g_1,\dots,g_n)$ is unique and depends only on $a$ and $\vec \lambda$, hence the following definition.
\end{rem}

\begin{defn}\label{df_Tplus}
 For each finite $\hat R$-independent tuple $\vec \lambda$, we add a new unary predicate $G_{\vec \lambda}$. For each finite $\hat R$-independent tuple $\vec \lambda = \lambda_1,\dots, \lambda_n$ and $1\leq i\leq n$ we also add new a unary function symbol $f_{\vec \lambda,i}$. Let 
 \[\cL_G^+ = \cL_G\cup\set{G_{\vec \lambda}\mid \vec \lambda \in \FF^n \mbox{ is $\hat R$-independent}}\cup \set{f_{\vec \lambda,i}\mid , 1\leq i \leq n, \vec\lambda\in \FF^n \mbox{ is $\hat R$-independent}  }.\] Let $T^{G+}$ be the expansion of $T^G$ to the language $\cL_G^+$ by the following sentences:
\[\forall x \left( G_{\vec \lambda}(x) \leftrightarrow \exists \vec y \in G \ \sum_i \lambda_i y_i = x\right)\]
\[\forall x, y \left[x = f_{\vec\lambda,i}(y) \leftrightarrow \left(y\in G_{\vec\lambda} \wedge \exists \vec x \in G\ y = \vec \lambda \cdot \vec x \wedge x_i = x \right) \vee \left(y\notin G_{\vec\lambda}\wedge x = 0\right)\right]\]
We will show that $T^{G+}$ has quantifier elimination. 
\end{defn}

The following notion is a straightforward modification of the corresponding notions from \cite{vdDG} and \cite{BHK}.

\begin{defn}[Mordell-Lang property]
A model $(V,G)$ of $T_U$ has the \emph{Mordell-Lang property} if for all definable sets $X$ of the form $\lambda_1 x_1+\dots +\lambda_n x_n=0$ with $\lambda_1,\dots,\lambda_n\in \FF$, the trace $X\cap G^n$ is equal to the trace $Y\cap G^n$ for $Y$ defined by a conjunction of formulas of the form $r_1x_1+\ldots+r_nx_n=0$ where $r_i\in R$. In particular, $X\cap G^n$ is $\emptyset$-definable from the $R$-module structure in $G$, i.e., by a formula of the form $P_\psi(x_1,\dots,x_n)$ for some $\cL_{R-mod}$ pp-formula $\psi$.
\end{defn}

\begin{lem}\label{lm_ML}
Every model of $T_U$ has the Mordell-Lang property.
\end{lem}
\begin{proof}
Let $\lambda_0,\lambda_1,\ldots, \lambda_k\in \mathbb{F}$ and let $X$ be the definable set given by \[\lambda_0 x_0+\lambda_1x_1+\ldots +\lambda_kx_k=0.\]
We may assume that for some $0\le m\le k$, $\lambda_0,\lambda_1,\ldots, \lambda_m$ are linearly independent over $\hat R$ as elements of the field $\mathbb{F}$, while $\lambda_{m+1},\ldots,\lambda_k\in \spane_{\hat R}(\lambda_0,\lambda_1,\ldots, \lambda_m)$.

Then for each $m+1\le i\le k$ we have \[\lambda_i=\Sigma_{j=0}^m q_{i,j}\lambda_j,\]
where $q_{i,j}\in \hat R$. Now let $g_0,\dots,g_k\in G$ be a realization of $X$. Collecting terms with $\lambda_0,\ldots, \lambda_m$, we have:
\[\lambda_0(g_0+\Sigma_{i=m+1}^k q_{i,0}g_i)+\lambda_1(g_1+\Sigma_{i=m+1}^k q_{i,1}g_i)+\ldots+ \lambda_m(g_m+\Sigma_{i=m+1}^k q_{i,m}g_i)=0.\]

By multiplying by a common multiple of the denominators of the $q_{i,j}$, there are $r_{i,j}\in R$ such that 
\[\lambda_0(r_{0,0}g_0+\Sigma_{i=m+1}^k r_{i,0}g_i)+\lambda_1(r_{1,1}g_1+\Sigma_{i=m+1}^k r_{i,1}g_i)+\ldots+ \lambda_m(r_{m,m}g_m+\Sigma_{i=m+1}^k r_{i,m}g_i)=0.\]

By axiom (B), each of the terms in parentheses above is equal to $0$, that is \[r_{0,0}g_0+r_{m+1,0}g_{m+1}+r_{m+2,0}g_{m+2}+\ldots+ r_{k,0}g_k=0.\]
\[\vdots\]
\[r_{m,m}g_m+r_{m+1,m}g_{m+1}+r_{m+2,m}g_{m+2}+\ldots+ r_{k,m}g_k=0.\]
The collection of equations listed (as a formula in the elements $g_0,\dots,g_k$) is  $\emptyset$-definable in the $R$-module structure in $G$. It is easy to see also that any solution $(g_0,\dots,g_k)\in G^{k+1}$ of this set of equations is also a realization of $X\cap G^{k+1}$ and vice versa.
\end{proof}

\begin{rema}\label{rk:mordelllang}
We actually proved a stronger version of the Mordell-Lang property: given a definable set $X$ of the form $\lambda_1 x_1+\dots +\lambda_n x_n=0$ with $\lambda_1,\dots,\lambda_n\in \FF$, the trace $X\cap G^n$ is quantifier free definable by a positive formula in the $R$-module $G$. We could have also restated the Modell-Lang property as any $\FF$-linear dependence in $G$ is witnessed by a linear combination with coefficients in $\hat R$ 
\end{rema}

Recall that $\cL_0 = \set{(\lambda\cdot)_{\lambda\in \FF}, +,0}$. We want to express $\cL_G^+$-terms using only the $\FF$-vector space language together with the $f_{\vec \lambda,i}$-functions. So we introduce the auxiliary language
\[\cL_0^+ := \cL_0\cup\set{ (f_{\vec \lambda,i})_{\vec \lambda\in \FF^n, i\leq n}}.\]

\begin{lem}\label{lm_reduction_fG} Let $(V,G)\models T^{G+}$ and
  let $t(\vec x)$ be an $\cL_G^+$-term. Then there exists $\cL$-formulas $\theta_1(\vec x),\dots,\theta_n(\vec x)$ forming a partition of $V^{|\vec x|}$ and $\cL_0^+$-terms $t_1(\vec x), \dots, t_n(\vec x)$ such that
  \[t(\vec x ) = y \leftrightarrow \bigvee_{i=1}^n t_i(\vec x)=y\wedge \theta_i(\vec x).\]
\end{lem}
\begin{proof}
  We prove it by induction on the complexity of $t(\vec x)$. If $t(\vec x) = g(t'_1(\vec x), \dots, t'_n(\vec x))$ for some $\cL$-function $g$, then by induction hypothesis, for each $1\leq k\leq n$ there exists $\cL$-partitions $(\theta_i^k(\vec x))_i$ and $\cL_0^+$-terms $(t_i^k(\vec x))_i$ such that for each $k\leq n$, $t'_k(\vec x) = z \leftrightarrow \bigvee_i t_i^k(\vec x)=z\wedge \theta_i^k(\vec x)$. Also, by Lemma~\ref{disj_terms}, there exists $\cL_0$-terms $s_1(\vec z), \dots, s_m(\vec z)$ and an $\cL$-partition $\psi_1(\vec z),\dots, \psi_m(\vec z)$ such that $g(\vec z) = y \leftrightarrow \bigvee_j s_j(\vec z)=y\wedge \psi_j(\vec z)$. Putting together, we have 
  \begin{align*}
  t(\vec x) = y &\leftrightarrow \left( \exists z_1,\dots, z_n \bigwedge_k z_k = t_k'(\vec x) \wedge g(\vec z) = y \right)\\
                &\leftrightarrow \left( \exists z_1,\dots, z_n \bigwedge_k \left( \bigvee_i t_i^k(\vec x)=z_k \wedge \theta_i^k(\vec x) \right) \wedge \left( \bigvee_j s_j(\vec z)=y\wedge \psi_j(\vec z) \right) \right)
\end{align*}
It is easy to see that for any finite family of sets $A_{i,k}$ such that $A_{i,k}\cap A_{j,k} = \emptyset$ for $i\neq j$, then $\bigcap_k\bigcup_i A_{i,k} = \bigcup_i\bigcap_k A_{i,k}$ hence 
\begin{align*}
     t(\vec x) = y           &\leftrightarrow \left( \exists z_1,\dots, z_n \bigwedge_k \bigvee_{i,j} \left( t_i^k(\vec x)=z_k \wedge \theta_i^k(\vec x) \wedge  s_j(\vec z)=y\wedge \psi_j(\vec z) \right)\right)\\
                &\leftrightarrow \left( \exists z_1,\dots, z_n  \bigvee_{i,j} \bigwedge_k \left( t_i^k(\vec x)=z_k \wedge \theta_i^k(\vec x) \wedge  s_j(\vec z)=y\wedge \psi_j(\vec z) \right)\right)\\
                &\leftrightarrow \left( \bigvee_{i,j}  \bigwedge_k \left( s_j(t_i^1(\vec x), \dots, t_i^n(\vec x))=y\wedge \theta_i^k(\vec x) \wedge \psi_j(t_i^1(\vec x), \dots, t_i^n(\vec x)) \right)\right)\\
                &\leftrightarrow  \bigvee_{i,j}  s_j(t_i^1(\vec x), \dots, t_i^n(\vec x))=y\wedge \left(\left(\bigwedge_k  \theta_i^k(\vec x)\right) \wedge \psi_j(t_i^1(\vec x), \dots, t_i^n(\vec x)) \right)
 \end{align*}
We now have to check that the family $\left(\left(\bigwedge_k  \theta_i^k(\vec x)\right) \wedge \psi_j(t_i^1(\vec x), \dots, t_i^n(\vec x)) \right)_{i,j}$ forms a partition of $V^{\abs{\vec x}}$. It is clear that the union is the whole universe. Let $(i,j)\neq (i',j')$. If $i\neq i'$, then $\left(\bigwedge_k  \theta_i^k(\vec x)\right) \wedge \psi_j(t_i^1(\vec x), \dots, t_i^n(\vec x))$ and $\left(\bigwedge_k  \theta_{i'}^k(\vec x)\right) \wedge \psi_{j'}(t_{i'}^1(\vec x), \dots, t_{i'}^n(\vec x))$ are disjoint since $(\theta_i^k(\vec x))_i$ is a partition for all $k$. Otherwise if $i=i'$ and $j\neq j'$ then the result follows from the fact that $\psi_{j}(t_{i}^1(\vec x), \dots, t_{i}^n(\vec x))$ and $\psi_{j'}(t_{i}^1(\vec x), \dots, t_{i}^n(\vec x))$ are disjoint, because $(\psi_j(\vec z))_j$ defines a partition of $V^{\abs{\vec z}}$.
\end{proof}

We now give a description of $\cL_0^+$-terms.

 \begin{lem}\label{lm_termsall}
 In a model $(V,G)$ of $T^{G+}$, any term $t(\vec a)$ in the language $\cL_0^+$ is equivalent to a term of the form:
 \[ \sum_i\lambda_i f_{\vec \mu_i, k_i}(\vec \alpha_i \cdot \vec a')+\vec \beta \cdot \vec a\]
 where $\vec a ' = \vec a\cap (V\setminus G)$ and $\lambda_i\in \FF, \vec \alpha_i\in \FF^{|\vec a'|}, \vec \mu_i \in \FF^m, \beta\in \FF^{|\vec a|}$.
 \end{lem}
 \begin{proof}

 First, observe the following:\\
 \textit{Claim}: For every $g\in G$, $b\in V$, $\alpha,\vec \mu\in \FF$, there exists $s\in R$, $q\in \hat R$ and $\vec\mu'\in \FF$ such that 
 \[f_{\vec \mu,k}(\alpha g +b) = \frac{1}{s} f_{\vec \mu',k}(sb)+ qg.\]

 We proof the claim by cases.
 \begin{itemize}
     \item Case 1: Assume $\alpha\notin \spane_{\hat R}(\vec \mu)$. Then $f_{\vec \mu,k}(\alpha g +b) = f_{\vec \mu',k}(b)$, for $\vec \mu ' = \vec \mu^{\frown}(-\alpha)$, so choose $s = 1$ and $q = 0$;
    \item Case 2: Assume $\alpha \in \spane_{\hat R}(\vec \mu)$, write $\alpha = q_1\mu_1+\dots+ q_n\mu_n$ for some $q_j\in \hat R$. Let $s$ be the product of denominators of $(q_i)_i$, so $sq_i\in R$, for all $i$. Then write $b+\alpha g = \vec \mu \cdot \vec g$ with $\vec g\in G$, so $f_{\vec\mu,k}(b+\alpha g) = g_k$. On the other hand $b = \sum_i \mu_i(g_i-q_i g)$, hence $sb = \sum_i\mu_i (sg_i+sq_i g)$. As $sg_i+sq_i g\in G$, we have $f_{\vec\mu,k}(sb) = sg_k+rg$ for $r = sq_k\in R$. We conclude that $f_{\vec \mu,k}(\alpha g +b) = g_k =  \frac{1}{s} f_{\vec \mu,k}(sb) - q_k g$, so choose $q = q_k$.
 \end{itemize}

 By the claim, the result follows if every term is equivalent to one of the form $\sum_i\lambda_i f_{\vec \mu_i, k_i}(\vec \alpha_i \cdot \vec a)+\vec \beta \cdot \vec a$, which we prove now by induction. By linearity of the expression, the only step to check is the one where $t(\vec a) = f_{\vec \gamma,i}(t'(\vec a))$. For convenience, we assume $i = 1$. By induction, $t'(\vec a)$ is of the form $\sum_i\lambda_i f_{\vec \mu_i, k_i}(\vec \alpha \cdot \vec a)+\vec \beta \cdot \vec a$. If $t(\vec a)\neq 0$, then there exists $g_1,\dots,g_n$ such that 
 \[\sum_i\lambda_i f_{\vec \mu_i, k_i}(\vec \alpha_i \cdot \vec a)+\vec \beta \cdot \vec a = \vec \gamma\cdot \vec g\]
 As $\vec \gamma$ is linearly independent, we may assume there is $s<\abs{\vec \lambda}$ be such that $\vec \gamma ^\frown(\lambda_i)_{i\leq s}$ is $\hat R$-independent and $\lambda_i\in \spane_{\hat R}(\vec \gamma (\lambda_i)_{i\leq s})$ for $i>s$. Let $q_{i,j}, q_{i,j}'\in \hat R$ be such that for $j>s$ we have $\lambda_j = \sum_{i=0}^s q_{i,j}\lambda_i+ \sum_{i=0}^{\abs{\vec \gamma}} q_{i,j}'\gamma_i$. Let $\vec \delta = \vec \gamma^\frown(\lambda_i)_{i\leq s}$. It follows that \[f_{\vec \delta,1}(\vec \beta \cdot \vec a) = g_1 - \sum_{j = s+1}^{\abs{\vec \lambda}} q_{1,j}f_{\vec\mu_j,k_j}(\vec \alpha_j \cdot \vec a)\]
 So $g_1 = t(\vec a)$ is of the required form.\end{proof}
 
 We denote by $\vect{B}$ the $\cL_0^+$-substructure generated by $B$. Using Lemma \ref{lm_termsall}, 
 \[\vect{B} = \spane_\FF(B\cup\set{f_{\vec{\lambda},i}(c) \mid c\in \spane_\FF(B), \vec\lambda\in \FF, i\leq \abs{\vec\lambda}}).\]

 \begin{lem}\label{lm_termscases}
 Let $(V,G)$ be a model of $T^{G+}$, $B\subset V$ and $a\in V$.
\begin{enumerate}
    \item If $a\in G$ or $a\notin\spane_\FF(BG)$ then every $\cL_0^+$-term $t(a,\vec b)$ with $\vec b\in B$ is equal to a term of the form $\lambda a +b'$ for some $b'\in \vect{B}$ and $\lambda\in \FF$.
    \item If $a\in\spane_\FF(BG)$, let $\vec \alpha\in \FF$ be $\hat R$-independent and let $c\in \spane_\FF(B)$ be such that $a = \sum_{i=1}^n \alpha_if_{\vec \alpha,i}(a-c) +c$. Then every $\cL_0^+$-term $t(a, \vec b)$ for $\vec b\in B$ is equal to a term of the form $\sum_{i=1}^{\abs{\alpha}}\lambda_i f_{\vec\alpha,i}(a - c)+b'$ for $\vec\lambda,\vec\beta\in \FF$, $b'\in \vect{B}$.
\end{enumerate}
\end{lem}

\begin{proof}
By Lemma \ref{lm_termsall}, every $\cL_0^+$-term $t(a,\vec b)$ is equal to $\sum_i\lambda_i f_{\vec \mu_i, k_i}(\vec \alpha_i \vec c)+\beta a+ \vec \beta \cdot \vec b$, for $\vec c = a^{\frown}\vec b \cap (V\setminus G)$.\\
\textit{(1)} Assume first that $a\in G$. The result is clear since $\vec c \subseteq \vec b$, so we can choose $b' = \sum_i\lambda_i f_{\vec \mu_i, k_i}(\vec \alpha_i \vec c)+\vec \beta \cdot \vec b\in \vect{B}$ and $t(a,\vec b)=b'+\beta a$. If on the other hand $a\notin \spane_\FF(BG)$, then $f_{\vec\mu,i}(\alpha a+\vec \alpha\vec b) = 0$ for all $\alpha, \vec\alpha, \vec \mu$, so the result follows.\\
\textit{(2)} Let $\alpha_1,\dots,\alpha_n\in \FF$ be such that for some $g_1,\dots,g_n\in G$ one have $a = \vec\alpha\cdot\vec g +c$ for $c\in \spane_\FF(B)$. By extracting an $\hat R$-basis of $\vec \alpha$ and replacing $a$ by $ra$ for some $r\in R$, we may assume that $\vec \alpha$ is $\FF$-independent, so $a = \sum_{i=1}^n \alpha_if_{\vec \alpha,i}(a-c) +c$ and $g_i = f_{\vec \alpha,i}(a-c)$. 

\textit{Claim:} For all $\beta,\vec \mu \in \FF$ and $b\in \vect B$, $f_{\vec\mu,k}(\beta a + b) = b'+ \sum_{i=1}^n q_i f_{\vec \alpha,i}(a-c)$, for some $b'\in \vect B$, and $q_i\in \hat R$.

For all $\vec \mu,\beta\in \FF$ and $b\in \vect{B}$, $f_{\vec\mu,k}(\beta a + b)= f_{\vec\mu,k}(\alpha_1'g_1+\dots+\alpha_n' g_n + \beta c + b)$ for $\alpha_i' = \beta \alpha_i$. From the claim in the proof of Lemma \ref{lm_termsall}, we have that 
\begin{align*}
    f_{\vec\mu,k}(\beta a + b) &= f_{\vec\mu,k}(\alpha_1'g_1+\dots+\alpha_n' g_n +\beta c+b)\\ &= f_{\vec\mu',k}( \beta c+b)+\sum_{i=1}^n q_i g_i\\
    &= b'+ \sum_{i=1}^n q_i f_{\vec \alpha,i}(a-c) 
\end{align*}
for $b'= f_{\vec\mu',k}( \beta c+b)\in \vect B$.

By Lemma \ref{lm_termsall}, for all $\vec b\in B$, every $\cL_0^+$-term $t(a,\vec b)$ is the sum of an $\FF$-linear combinations of $a,\vec b$ (which is of the desired form since $a = \sum_{i=1}^n \alpha_if_{\vec \alpha,i}(a-c) +c$) and of an $\FF$-linear combination of $f_{\vec \mu, k}(\beta a +\vec \gamma \vec b)$ which is also of the desired form by the Claim.
\end{proof}

We will use the term `locally' to mean that something holds on a finite definable partition of the universe.
\begin{lem}\label{cor_QEmod}
Let $(V,G)$ be a model of $T^{G+}$, let $\vec b\in V$, $\vec \lambda\in \FF$. Let $t(\vec x,\vec y)$ be an $\cL^+_G$-term. Then every formula of the form 
\[t(\vec x,\vec b) \in G_{\vec \lambda}\wedge \vec x\subset G\]
is locally equivalent to a formula of the form $P_{\psi}(\vec x, \vec h)$ for $\psi(\vec x,\vec z)$ an $\cL_{R-mod}$-formula and a tuple $\vec h$ in $G(\vect{B})$. Equivalently, there is a finite $\cL$-partition $(\theta_i(\vec x))_i$ of the universe and $\cL_{R-mod}$-formulas $\psi_i$ such that 
\[\left(t(\vec x,\vec b) \in G_{\vec \lambda}\wedge \vec x\subset G\right) \leftrightarrow \bigvee_i P_{\psi_i}(\vec x, \vec h)\wedge \theta_i(\vec x).\]
\end{lem}
\begin{proof}
By Lemma \ref{lm_reduction_fG}, $t(\vec x,\vec b)$ is locally an $\cL_0^+$-term, and by Lemma~\ref{lm_termscases} \textit{(1)}, we may assume that $t(\vec x,\vec b)$ is locally of the form $\vec \alpha\cdot  \vec x + b$ for some $b\in \vect B$. We show that the formula 
$\vec \alpha\cdot  \vec x + b\in G_{\vec \lambda}$ is equivalent to a formula $P_\psi(\vec x, \vec h)$ for $\vec h\in G(\vect B)$ and $\psi$ an $\cL_{R-mod}$-formula. Let $\vec a\in G$ be a realisation of $\vec \alpha\cdot  \vec x + b\in G_{\vec \lambda}$ and let $\vec g'\in G$ be such that $\vec\alpha\cdot \vec a + b = \vec \lambda\cdot \vec g'$. Then for some basis $\vec \delta$ of $\spane_{\hat R}(\vec\alpha\vec\lambda)$ we have $b\in G_{\vec\delta}$ and hence there exists $\vec h\in G(\vect B)$ such that $b = \vec\delta\cdot\vec h$. It follows that $\vec\alpha\cdot \vec a + \vec\delta\cdot \vec h = \vec \lambda\cdot \vec g'$. As $\vec a\in G$, by Lemma \ref{lm_ML}, there exists an $\cL_{R-mod}$-formula $\psi(\vec x, \vec y)$ such that $\exists \vec z\in G (\vec\alpha\cdot \vec x + \vec\delta\cdot\vec h = \vec \lambda\cdot \vec z)$ is equivalent to $G\models \psi(\vec x, \vec h)$, hence to $P_\psi(\vec x, \vec h)$.
\end{proof}

\begin{rem}\label{rk_cor_term}
Note that the proof above also proves that  whenever $t(\vec x,\vec b)$ is a term, there exists a finite $\cL$-partition $(\theta_i(\vec x))_{i\leq n}$
of the universe and $\cL_{R-mod}$-formulas $(\psi_i)_{i\leq n}$ such that

\[\left(t(\vec x,\vec b)=0 \wedge \vec x\subset G\right) \leftrightarrow \bigvee_i P_{\psi_i}(\vec x, \vec h)\wedge \theta_i(\vec x).\]
\end{rem}

\begin{prop}\label{prop_shift}
Let $(V,G)$ be a model of $T^{G+}$ and let $B = \vect B\subset V$. Assume that $(V,G)$ is $\abs{B}$-saturated. Let $a\in G$ and let \begin{itemize}
    \item $q(x)$ be the set of boolean combinations of formulas of the form $t(x, \vec b)\in G_{\vec \lambda}$ satisfied by $a$, for $\vec b\in B$, $t( x, \vec b)$ an $\cL_G^+$-term and $\vec \lambda\in\FF$.
    \item $q_1(x)$ be the set of formulas of the form $P_\phi(x,\vec g)$ satisfied by $a$, for $\vec g\in G(B)$, where $\phi(x,\vec y)$ is an $\cL_{R-mod}$-formula.
\end{itemize}  
Then  $\tp_\cL(a/\emptyset)\cup q_1(x)\models q(x)$. Moreover, for any non-algebraic $\cL$-type $p(x)$, $p(x)\cup q_1(x)$ is consistent with infinitely many realisations.
\end{prop}
\begin{proof}
Let $\phi(x,\vec b)\in q(x)$. By Lemma~\ref{cor_QEmod}, for $ x\in G$, every formula of the form $t(x,\vec b)\in G_{\vec \lambda}$ is equivalent to $\bigvee_i P_{\psi_i}(\vec x, \vec h)\wedge \theta_i(\vec x)$. As $\cL$-formulas and formulas of the form $P_\psi$ are closed by boolean combinations, every formula $t(x,\vec b)\not \in G_{\vec \lambda}$ is also equivalent to a finite disjunction of expressions the form $\bigvee_j P_{\psi_j}(x,\vec h) \wedge \vartheta_j(x)$ (note that here $\vartheta_j$ do not necessarily form a partition of the universe). More generally, for $\phi(x,\vec b)\in q(x)$, the expression $\phi(x,\vec b) \wedge x\in G$ is equivalent to a formula of the form $\bigvee_j P_{\psi_j}(x,\vec h) \wedge \vartheta_j(x)$, for $\cL_{R-mod}$-formulas $\psi_i$ and $\cL$-formulas $\vartheta_i$. Let $a'\models \tp_\cL(a/\emptyset)\cup q_1(x)$  and let $\phi(x,\vec b)\in q(x)$. Then by Lemma \ref{cor_QEmod} there is an $\cL_{R-mod}$-formula $\psi(x,\vec y)$ and an $\cL$-formula $\vartheta(x)$ such that $a\models  P_{\psi}(x,\vec g)\wedge \vartheta(x)$ and $ (P_{\psi}(x,\vec g)\wedge\vartheta(x))\rightarrow \phi(x,\vec b)$, for some $g\in G(B)$. It follows that $a'$ satisfies $\phi(x,\vec b)$.

Let $p(x)$ be a non-algebraic $\cL$-type. We show that $q_1(x)$ is consistent with $p(x)$, with infinitely many realisations.  Let $p_{sh}(x) = p(x+a)$. The type $p_{sh}$ is also non-algebraic, hence by density of $R$-divisible elements (Lemma \ref{lm_densediv}), there exists infinitely many $d$'s such that $d\models p_{sh}(x)\cap G^\dv$. Let $a'=d+a$.

Claim: $a'\models q_1 (x)\cup p(x)$. First, as $d\models p_{sh}(x)$, $a' = d+a\models p(x)$. By quantifier elimination in $R$-modules, every formula in $q_1(x)$ is a boolean combination of conditions of the form $ra'+b \in s_1G+\dots+s_nG$ for $r,\vec s \in R$ and $b\in G(B)$. As $d\in G^\dv$, we have that $rd\in s_1G+\dots+s_nG$ for all $r,\vec s\in R$, hence $ra'+b \in s_1G+\dots+s_nG$ if and only if $ra+b\in s_1G+\dots+s_nG$, so $a'\models q_1(x)$. It follows that $q_1(x)\cup p(x)$ has infinitely many realisations.
\end{proof}

\begin{cor}\label{cor_shift}
Let $p(\vec x)$ be an $\cL$-type over $B = \vect B$, such that \[p(\vec x)\models ``\text{$\vec x$ is $\FF$-independent over $B$}".\] 
Let $q(\vec x)$ be any set of boolean combinations of formula of the form $t(\vec x, \vec b)\in G_{\vec \lambda}$, for $\vec b\in B$ and $t(\vec x,\vec y)$ an $\cL^+_G$-term, which is consistent in $G$. If $(p(\vec x)\upharpoonright \emptyset) \cup q(\vec x)$ is consistent in $G$, then $p(\vec x)\cup q(\vec x)$ is also consistent in $G$.
\end{cor}
\begin{proof}
We prove it by induction on $\abs{\vec x}$. For $\abs{\vec x} = 1$, it is Lemma \ref{prop_shift} (since $p(x)\cup q_1(x)\models p(x)\cup q(x)$). 
Assume $\vec x = (x_1,\dots,x_n)$. Let $p_{n-1}(x_1,\dots, x_{n-1})$ and $q_{n-1}(x_1,\dots, x_{n-1})$ be the restrictions of $p(\vec x)$ and $q(\vec x)$ to the first $n-1$ variables. By induction hypothesis, there exists $(a_1,\dots,a_{n-1})\models p_{n-1}\cup q_{n-1}$. In particular, $a_1,\dots, a_{n-1}$ are $\FF$-independent over $B$. Let $p'(x_n)$ and $q'(x_n)$ be completions to $\vect{Ba_1,\dots,a_{n-1}}$ of $p(a_1,\dots,a_{n-1},x_n)$ and $q(a_1,\dots,a_{n-1},x_n)$, so that $p'(x_n) \models ``x_n\notin \vect{Ba_1,\dots,a_{n-1}}"$. Using again Lemma \ref{prop_shift}, $p'(x_n)\cup q'(x_n)$ has a realisation, say $a_n$, then $(a_1,\dots a_n)\models p(\vec x)\cup q(\vec x)$. 
\end{proof}

\begin{thm}\label{thm_QE}
 The theory $T^{G+}$ has quantifier elimination.
\end{thm}

\begin{proof}
We show that the set of partial isomorphisms between two $\abs{\FF}^+$-saturated models $(V,G), (V', G')$ of $T^{G+}$ has the back and forth property. Let $B\subset V, B'\subset V'$ be two small substructures (i.e. $B = \vect B$ and $B' = \vect{B'}$) such that there exists a partial isomorphism $\sigma: B \rightarrow B'$ (i.e. a function that preserves the quantifier free type in $\cL_G^+$, relatively to the ambient model, $\tp^{(V,G)}_{QF\cL_G^+}(B) = \tp^{(V',G')}_{QF\cL_G^+}(B)$). 

 Let $a\in V\setminus B$. As $B=\vect{B}$, by hypothesis on $T$ we have that $\tp_\cL(a/B)$ is non-algebraic.

\underline{Step 1.} If $a\in G(V)$. From Proposition~\ref{prop_shift}, the quantifier-free type of $a$ over $B$ is implied by $\tp_\cL(a/B)$, the set $q(x)$ of boolean combinations of conditions of the form $t(x,\vec b)\in G_{\vec \lambda}$ satisfied by $a$ in $V$, for $b\in B$, $t(x,\vec y)$ an $\cL_G^+$ term and $\vec \lambda \in \FF$, and inequations $\alpha x \neq b$ for $b\in B$. Equations do not appear as $a\notin B$ (by Lemma~\ref{lm_termscases} (1)). By compactness and Proposition \ref{prop_shift}, it is enough to show that $\sigma[q(x)\cup \tp_\cL(a/B)]$ has infinitely many solutions. Let $q_1(x)$ be as in Proposition \ref{prop_shift}. In particular, for any $\phi(x,\vec g)\in q_1(x)$ we have $(V,G)\models P_\psi(\vec g)$ and $\vec g\in G(B)$ for $\psi(\vec y) = \exists x\phi(x,\vec y)$, so $(V',G')\models P_\psi(\sigma(\vec g))$. By compactness, this shows that $\sigma (q_1) = q_1^\sigma(x)$ is a consistent partial type in $(V',G')$, and $q_1^\sigma$ is the set of boolean combinations of formulas of the form $t(x,\vec b')$ satisfied by a realisation of $q_1^\sigma$ in $V'$.

Let $p(x) = \sigma(\tp_\cL(a/B))$. The fact that $q_1^\sigma(x)\cup p(x)$ has infinitely many realisations follows from Proposition \ref{prop_shift}, as $p(x)$ is non-algebraic. Using again Proposition \ref{prop_shift}, we have that $q_1^\sigma(x)\cup p(x)\models q^{\sigma}(x)$, so the type $q^\sigma(x)\cup p(x)$ is realised and non-algebraic, so we can extend $\sigma$ by $\sigma(a) = a'$ for some $a'\models q^\sigma(x)\cup p(x)$.

\underline{Step 2.} If $a\in \spane_\FF(B G)$, then by Lemma \ref{lm_termscases} (2) there exists $\hat R$-independent $\vec \alpha\in \FF$ and $c\in B$ such that $a =  \vec\alpha\cdot \vec g+ c$, and so $g_i = f_{\vec \alpha,i}(a-c)$. Lemma \ref{lm_termscases} (2) also implies that every $\cL_G^+$-term $t(a,\vec b)$ is locally an $\FF$-linear combination of $g_1,\dots,g_n$. By Step 1 we can extend $\sigma$ to $g_1,\dots,g_n$ hence we can extend $\sigma$ to $a$. 

\underline{Step 3.} If $a\in V\setminus \spane_\FF(B G)$.
By Lemma \ref{lm_termscases} (1), every term $t(a,\vec b)$ is locally of the form $\alpha a +b$, for some $b\in B$. 
So it is enough to find in $V'$ an element $a'$ such that $a'+b'\notin G_{\vec\lambda}$ for all $b'\in B'$, $\vec\lambda\in \FF$, and such that $\sigma(\tp_\cL(a/B)) = \tp_\cL(a'/B')$. As $\tp_\cL(a/B)$ is non-algebraic, such $a'$ exists by axiom $(D)$ and $\abs{\FF}^+$-saturation.
\end{proof}

Recall that $T_U$ is the theory consisting of $T$ together with the schemes $(A)$ and $(B)$ and $T^G$ the theory consisting of adding $(A),(B),(C),(D)$, hence $T^G$ is the restriction of $T^{G+}$ to the language $\cL_G$.
\begin{cor}
If $T_U$ is inductive, then $T^G$ is model-complete and it is the model-completion of $T_U$.
\end{cor}
\begin{proof}
$T^{G+}$ has quantifier elimination by Theorem \ref{thm_QE}, and every function $f_{\vec \lambda,i}$ is existentially definable in the language $\cL_G$, to $T^G$ is model-complete. Note that the proof of  Lemma~\ref{lm_pregeom_SAP} did not use that $\hat R = \FF$, hence we have that $T_U$ has SAP. It remains to prove that every model of $T_U$ extends to a model of $T^G$, which is similar to the proof of Corollary ~\ref{TU-modelcompletion}.
\end{proof}

 We can now characterize the algebraic closure in the extended language.
 
\begin{cor}
Let $(V,G)$ be a model of $T^{G+}$ and let $B\subseteq V$. Then 
\[\acl_G(B) = \dcl_G(B) = \vect{B} = \spane_\FF(B\cup\set{f_{\vec{\lambda},i}(c) \mid c\in \spane_\FF(B), \vec\lambda\in \FF, i\leq \abs{\vec\lambda}}).\]
\end{cor}

\begin{proof}
It is clear that $\vect B\subseteq \dcl_G(B)\subseteq \acl_G(B)$. Let $a\notin \vect B$.\\
\textit{Case 1.} If $a\in G$, then from quantifier elimination $\tp_G(a/B)$ is determined by $\tp_\cL(a/B)\cup q(x)$ for $q(x)$ as in Proposition \ref{prop_shift}, and is non-algebraic from the conclusion of Proposition \ref{prop_shift}, since $\tp_\cL(a/B)$ is non-algebraic.\\
\textit{Case 2.} If $a\in \spane_\FF(BG(V))$. By Lemma \ref{lm_termscases} (2) there is an $\FF$-independent tuple $\vec \alpha$, $\vec g\in G$ and $c\in B$ such that $a = \vec \alpha\cdot \vec g + c$ and $g_i = f_{\vec \alpha, i}(a-c)$ and such that every term in $a\vec b$ is equal to an element in $\spane_\FF(\vec gB)$. By quantifier elimination, $\tp_G(a/B)$ is determined by $\tp_{G}(\vec g/B)\cup \set{g_1 = f_{\vec \alpha, 1}(x-c),\dots, g_n = f_{\vec \alpha, n}(x-c)}$. As $a\notin \vect B$, there is some $i$ such that $g_i\notin \vect B$, hence by case 1 $g_i\notin \acl(B)$ hence $a\notin \acl(B)$.\\
\textit{Case 3.} If $a\in \spane_\FF(BG(V))$.
By Lemma \ref{lm_termscases} and quantifier elimination, $\tp_G(a/B)$ is determined by $\tp_\cL(a/B)\cup \set{x\notin G_{\vec \lambda}+b \mid b\in \vect B, \vec\lambda\in \FF, \hat R\text{-independent}}$. By codensity (D), as $\tp_\cL(a/B)$ is non algebraic, this type has infinitely many realisations, so $a\notin \acl_G(B)$.
\end{proof}

From Theorem~\ref{thm_QE}, every $\cL_G$-formula $\phi(\vec x)$ is a boolean combination of quantifier free $\cL_G^+$-formulas, so there exists an $\cL$-formula $\psi(\vec x)$, an $\cL_{R-mod}$-formula $\theta(\vec x)$ and $\cL_G^+$-terms $t_i(\vec x), t_i'(\vec x), t_j(\vec x), t_j'(\vec x)$ such that $\phi(\vec x)$ is equivalent to a disjunction of formulas of the form
 \[ \psi(\vec x)\wedge P_\theta (\vec x) \wedge \bigwedge_i t_i(\vec x)\in G_{\vec \lambda_i}\wedge \bigwedge_j t_j(\vec x)\notin G_{\vec \mu_j} \wedge \bigwedge_i t_i(\vec x)=0\wedge \bigwedge_j t_j(\vec x)\neq 0 \]
Note that $P_\theta(\vec x)$ is equivalent to a formula such that quantifiers only occur in the predicate $G$. Similarly, using the definition of $G_{\vec\lambda_i}$ and $f_{\vec\lambda,i}$, we see that quantification only occurs in the group $G$. This particular instance of formulas is called \emph{bounded} in the sense of \cite{ChSi}, hence every formula in $T^G$ is bounded.
 
 For subsets of $G$ we get a cleaner description.

\begin{defn}\label{def-G-ind}
Let $A\subset V$. We say that $A$ is {\it $G$-independent} if $A\ind_{G(A)}G(V)$.
\end{defn}

We now give a characterization of $G$-independent sets.

\begin{lem}\label{G-ind-characterization}
Let $A\subset V$, then $A$ is $G$-independent if and only if 
$\spane_{\FF}(A)=\vect A $.
\end{lem}

\begin{proof} 
Clearly $\spane_{\FF}(A)\subseteq \vect A$ holds for any set $A$. In order to proof the lemma, we will show that $A\ind_{G(A)} G(V)$ if and only if $\spane_{\FF}(A)\supseteq \vect A$.

We may assume that $A=\spane_{\FF}(A)$.
Suppose first that $A\nind_{G(A)} G(V)$, then there is $a\in A\setminus  \spane_{\FF}(G(A))$ such that $a\in \spane_{\FF}(G(V))$. Let $\lambda_1,\dots,\lambda_n\in \FF$ and $g_1,\dots,g_n\in G(V)$ be such that
$a=\lambda_1g_1+\dots+\lambda_ng_n$. We may choose $\lambda_1,\dots,\lambda_n\in \FF $ to be $\hat R$-independent. Since 
$a\not \in \spane_{\FF}(G(A))$, then $A$ is not
closed under the function $f_{\vec \lambda,i}$ for some $i$ and thus 
$\spane_{\FF}(A)\subsetneq \vect A$.

Now assume that $A\ind_{G(A)} G(V)$ and $A=\spane_{\FF}(A)$.
 Using Lemma \ref{lm_termsall}, we get that $\vect{A} = \spane_\FF(A\cup\set{f_{\vec{\lambda},i}(c) \mid c\in \spane_\FF(A), \vec\lambda\in \FF, i\leq \abs{\vec\lambda}})$, so $A = \vect A$ if and only if $f_{\vec{\lambda},i}(a)\in A$ for all $a\in \spane_\FF(A)$, $\vec\lambda\in \FF$ a $\hat R$-independent tuple and $i\leq \abs{\vec\lambda}$. Let $\lambda_1,\dots,\lambda_n\in \FF$ be $\hat R$-independent, let $a\in A$ and consider the function $f_{\vec \lambda,i}(a)$.
If $f_{\vec \lambda,i}(a)=0$, since $A=\spane_{\FF}(A)$, we have that $0\in A$ as desired. If $f_{\vec{\lambda},i}(a)\neq 0$, then $a\in \spane_\FF(G) \cap A$, which equals $\spane_\FF(G(A))$ by assumption,
so $a\in \spane_\FF(G(A))$. Then there exists an $\hat R$-independent tuple $\vec \alpha$ such that $a=\sum_i\alpha_i f_{\vec \alpha,i}(a)$, with $f_{\vec \alpha,i}(a)\in G(A)$. By Lemma~\ref{lm_termscases} (2) (with $B = \emptyset)$, every $\cL_0^+$-term in $a$ is a linear combination of $f_{\vec \alpha,i}(a)$ hence belongs to $\spane_\FF(A)$ so $f_{\vec{\lambda},i}(a)\in G(A)$.
\end{proof}

The next corollary shows quantifier elimination for $G$-independent tuples down to $\cL\cup\{G\}$-formulas and $R-mod$-formulas (the latter restricted to elements of $G$), similar to Proposition 3.4 of \cite{BeVaG}.

\begin{cor}\label{G-ind-qe}
Assume $(V,G)$, $(W,G)$ are models of $T^{G+}$ and let $\vec a \in V$ and $\vec b\in W$ be $G$-independent tuples such that $tp_{\cL}(\vec a, G(\vec a))=\tp_{\cL}(\vec b, G(\vec b))$ and $\tp_{R-mod}(G(\vec a))=\tp_{R-mod}(G(\vec b))$. Then $\tp_{\cL_G^+}(\vec a)=\tp_{\cL_G^+}(\vec b)$.
\end{cor}

\begin{proof}
Note that the $\cL$-elementary map taking $\vec a$ to $\vec b$ extends uniquely to an elementary map $$\tau: \spane_{\mathbb{F}}(\vec a)\to \spane_{\mathbb{F}}(\vec b).$$
We claim that $$\tau (G(\spane_{\mathbb{F}}(\vec a)))= G(\spane_{\mathbb{F}}(\vec b)).$$
It suffices to show that $$\tau (G(\spane_{\mathbb{F}}(\vec a)))\subseteq  G(\spane_{\mathbb{F}}(\vec b)).$$
Indeed, if $\vec a=(a_1,\ldots, a_n)$, and $g\in G(\spane_{\mathbb{F}}(\vec a))$, then $g\in G(V)$ and $$g=\lambda_1a_1+\ldots +\lambda_na_n$$ for some $\lambda_i\in\mathbb{F}$. On the other hand, $\vec a$ is $G$-independent, so assuming for simplicity that $G(\vec a)=(a_1,\ldots, a_k)$ for some $k\leq n$, we also have $g=\alpha_1 a_1+\ldots + \alpha_k a_k$ for some $\alpha_i\in \mathbb{F}$. By the Mordell-Lang property, we may assume that $\alpha_i\in \hat R$. The fact that $$\alpha_1 a_1+\ldots +\alpha_k a_k\in G(V)$$ is witnessed by $\tp_{R-mod}(G(\vec a))$, and since 
$\tp_{R-mod}(G(\vec a))=\tp_{R-mod}(G(\vec b))$, we also have that $\tau(g)=\alpha_1 \tau(a_1)+\ldots+ \alpha_k \tau(a_k)\in G(W)$, as we wanted.
Since $\vec a$ and $\vec b$ are $G$-independent, Lemma \ref{G-ind-characterization} implies that $\spane_\FF(a) = \vect{a}$ and $\spane_\FF(b) = \vect{b}$ and by the observation 
 above $\tau(f_{\vec \lambda,i}(\vec a))=f_{\vec \lambda,i}(\vec b)$, so $\tau$ preserves the functions in the extended language $\mathcal{L}_{G}^+$ and by quantifier elimination (Theorem \ref{thm_QE}) we get that $\tau$ is an elementary map. 
\end{proof}

\begin{cor}\label{inducedstructG}
Let $\vec a$ be a $G$-independent tuple, let $\vec h=G(\vec a)$ and let $\varphi(\vec x,\vec y)$ be an $\cL_G^+$-formula. Then there is an $\cL$-formula $\psi(\vec x,\vec y)$ and an $\cL_{R-mod}$-formula $\theta(\vec x, \vec y)$ such that
$$\forall \vec x ([\varphi(\vec x,\vec a)\wedge G(\vec x)] \leftrightarrow [\psi(\vec x,\vec a)\wedge P_\theta (\vec x,\vec h) \wedge G(\vec x)]).$$
\end{cor}
 
\begin{proof} 
This is a direct consequence of quantifier elimination (Theorem \ref{thm_QE}), Lemma \ref{cor_QEmod} and Remark \ref{rk_cor_term}.
\end{proof}

Recall from \cite{BeVaGroups} that a unary expansion $(M,P)$ of a model $M$ of geometric theory $T$ satisfies Type Equality Assumption (TEA) if whenever $\vec a, \vec b, \vec c\in M$ are such that $\vec a$ is $P$-independent (i.e. $\vec a\ind_{P(\vec a)} P(M)$), $\vec b\ind_{\vec a}P(M)$, $\vec c\ind_{\vec a}P(M)$, and $\tp(\vec b\vec a)=\tp(\vec c\vec a),$ then $\tp_{P}(\vec b\vec a)=\tp_{P}(\vec c\vec a)$ (where $\tp_P$ refers to the type in the language expanded by the unary predicate symbol).
 
 The following corollary follows directly form Corollary \ref{G-ind-qe}.

\begin{cor}\label{TEA}
 Let  $(V,G)$ be a model  of theory $T^{G+}$. Then the reduct of $(V,G)$ to the language $\cL \cup \{G\}$ satisfies TEA.
\end{cor}

Note that the conclusion of TEA also holds for the full language $\cL_G^+$.

\begin{rema}
In the context of our construction, one can define the notion of "G-basis" analogous to the one used in \cite{BeVaG}. Namely, given a $G$-independent set $C$ and tuple $\vec a$ in $V$, we are looking for a "canonical" subset of $G(V)$ such that adding it to $\vec a C$ makes the set $G$-independent. We claim that the appropriate notion of  $GB(\vec a/C)$ is given by $$\vect{\vec aC}\cap G(V)\backslash \spane_{\hat R}(G(C)).$$
Indeed, $GB(\vec a/C)\cup \vec a\cup C$ is $G$-independent, and for any $\vec b\in G(V)$ such that $\vec a \ind_{C\vec b} G(V)$, we have $GB(\vec a/C)\subset \spane_{\hat R}(\vec b G(C))$.
\end{rema}

We will now explicitly show that $T^G$ is consistent by showing how to construct a model using the notion of $H$-structure \cite{BeVa-H}. We will assume the reader is familiar with the definition of $H$-structure, but we will not require any deeper knowledge of its theory. Let $(\mathcal{V},H)$ be a sufficiently saturated $H$-structure  and let $G(\mathcal{V})$ be the $R$-submodule of $(V,+,0)$ generated by $H(V)$. 
We will show that $Th(\mathcal{V}, G)=T^G$.

\begin{lem}\label{H-model-axioms-ABC} $G(V)$ is dense-codense in $\mathcal{V}$ in the sense of geometric structures, and $(\mathcal{V}, G)$ satisfies axioms $(A)$, $(C)$, $(D)$ of $T^G$.

\end{lem}

\begin{proof}
It is clear that $H(\mathcal{V})\subset G(\mathcal{V})\subset \acl(H(\mathcal{V}))$. Since $H$ structures satisfy the density and codensity property (see Definition 2.2 in \cite{BeVa-H}) then so does $G(V)$ 
and it follows that $(A), (C)$ and $(D)$ hold for $G(V)$.
\end{proof}

\begin{lem}\label{H-model-axiom-D}
Let $\lambda_1,\ldots,\lambda_k\in\mathbb{F}$ be linearly independent over $\hat R$, $g_1,\ldots, g_k\in G(V)$ and assume that $\lambda_1 g_1+\ldots+\lambda_k g_k=0$. Then $g_1=\ldots=g_k=0$, i.e. $(V,G)$ satisfies  axiom $(B)$ of $T^G$.
\end{lem}

\begin{proof}
 Let $h_1,\ldots, h_m\in H(V)$
be distinct elements such that $$g_1,\ldots,g_k\in \spane_R(h_1,\ldots,h_m).$$
For each $1\le i\le k$ let $r_{ij}\in R$ be such that $$ g_i=\sum_{j=1}^m r_{ij}h_j.$$

Then we have: $$ \lambda_1 \sum_{j=1}^m r_{1j}h_j+\ldots+\lambda_k \sum_{j=1}^m r_{kj}h_j=0  $$

 $$ \sum_{i=1}^k \sum_{j=1}^m \lambda_ir_{ij}h_j=0$$

$$ \left(\sum_{i=1}^k \lambda_i r_{i1}\right)h_1+\ldots+ \left(\sum_{i=1}^k \lambda_i r_{im}\right)h_m=0.  $$

Hence, for any $1\le j\le m$ we have  $\sum_{i=1}^k r_{ij} \lambda_i=0$. Since $\lambda_1,\ldots,\lambda_k$ are linearly independent over $\hat R$, we conclude that $r_{ij}=0$ for all $1\le i\le k$ and $1\le j\le m$. Thus, $g_1=\ldots=g_k=0$, as needed.

\end{proof}

We can now deduce the following result.

\begin{prop}\label{TG-consistent}
The theory $T^G$ is complete and consistent.
\end{prop}

\begin{proof}
Consistency follows by Lemmas \ref{H-model-axioms-ABC} and \ref{H-model-axiom-D}. Completeness follows by Theorem \ref{thm_QE}.
\end{proof}

\section{Preservation of stability and NIP}\label{sec:NIPandstable}

Now we will start by proving that NIP is preserved in the expansion by using the 
ideas of Chernikov and Simon \cite{ChSi}. Let $(V,G)$ be a $G$-structure. Recall \cite{ChSi} that a formula $\varphi(\vec x, \vec y) \in \mathcal{L}_G$ is said to be \emph{NIP over $G$} if there
is no $\mathcal{L}_G$-indiscernible sequence $\{\vec a_i: i\in \omega\}$
of tuples of $G$ and an element $\vec b\in M$ such that $\varphi(\vec a_i
, \vec b) \leftrightarrow i$ is even.

\begin{prop}\label{NIPoverG}
Assume that $T$ is NIP. Then no formula
$\varphi(\vec x, \vec y) \in \mathcal{L}_G$ is NIP over $G$.
\end{prop}

\begin{proof}
Assume otherwise, so there is a $\mathcal{L}_G$-indiscernible sequence $\{\vec a_i : i\in \omega\}$
of tuples of $G$ and an element $\vec b\in M$ such that $\varphi(\vec a_i
, \vec b) \leftrightarrow i$ is even. We may enlarge $\vec b$ if necessary and assume that $\vec b$ is $G$-independent and let $\vec h=G(\vec b)$. By
Corollary \ref{inducedstructG} there are
$\theta(\vec x)$ an $R$-module formula and $\psi(\vec x,\vec y)$ an $\mathcal{L}$-formula such that 
$\forall x (\varphi(\vec x,\vec b)\wedge G(\vec x) \leftrightarrow \theta(\vec x,\vec h) \wedge \psi(\vec x,\vec b))$.
But every $R$-module formula is stable, so the cofinal value of $\theta(\vec a_i,\vec h)$ is fixed. Similarly, since $T$ is NIP, the cofinal value of $\psi(\vec a_i,\vec b)$ is fixed, a contradiction.
\end{proof}

\begin{prop}
Assume that $T$ is NIP, then $T^{G+}$ is also NIP.
\end{prop}

\begin{proof}
By quantifier elimination in $T^{G+}$ every $\mathcal{L}_G$ formula is equivalent to a bounded formula and by Proposition \ref{NIPoverG} no formula has NIP over $G$, by Theorem 2.4 in \cite{ChSi} we have that $T_P$ is also NIP.
\end{proof}

A similar argument studying the induced structure of $(V,G)$ in $G$ and using the ideas of Casanovas and Ziegler \cite{CZ} shows that stability is also preserved in the pair:

\begin{prop}
Assume that $T$ is stable, then $T^{G+}$ is also stable.
\end{prop}

We will now consider some examples.

\begin{exam}\label{purevectorspace2}
(Pure vector spaces, example \ref{purevectorspace} revisited)

Let $\mathcal{V}$ be a pure vector space over a field $\mathbb{F}$, which is strongly minimal. It is proved in \cite{CZ} that the stability spectrum of the expansion depends on the stability spectrum of $\mathcal{V}$ and the stability spectrum induced by the pair in the predicate. In our setting, since $\mathcal{V}$ is $\aleph_0$-stable, it depends on the relation between $R$ and $\FF$. 

We will concentrate in the case where the pair is $\omega$-stable and we will construct expansions with Morley rank $n$ for every $n\geq 2$. Assume that $R$ and $\FF$ are fields and $[\FF:R]=n$. We can see $\FF$ as a vector space of dimension $n$ over $R$. Choose $\lambda_1,\dots,\lambda_n$ a basis of $\FF$ over $R$. Consider the map $f:G^n\to V$ given by $f(g_1,\dots,g_n)=
\lambda_1 g_1+\dots+\lambda_ng_n$. This map is definable and generically one to one. Since the structure of $G$ is that of a pure vector space over $R$, we have that $MR(G)=1$ and thus $MR(f(G^n))=n$. By the extension property $V$ contains properly the set $f(G^n)$, so $MR(\mathcal{V},G)=n+1$. 
Note that $f(G^n)$ is a vector space over $\FF$ and that the structure $(\mathcal{V},f(G^n))$ is a lovely pair of the theory vector spaces over $\FF$.
\end{exam}

\begin{lem}(Lemma \ref{1basedgroupscase1} revisited)
Let $\mathcal{V}$ be a pure vector space over a field $\mathbb{F}$ and let $R$ be a subring of $\FF$. Then $T^{G^+}$ is $1$-based.
\end{lem}

\begin{proof}
    We follow the same strategy as we did in section \ref{Sub:qefractionfield}. Since $\mathcal{V}$ is a $1$-based group and $T^{G^+}$ has quantifier elimination in the extended language, it is enough to show that atomic formulas define boolean combinations of cosets of $\emptyset$-definable subgroups. By the proof of Lemma \ref{1basedgroupscase1} it suffices to check that the result is true for formulas of the form $G_{\vec \lambda}(t(\vec x,\vec a))$ and $t(\vec x,\vec a)=0$ where $t(\vec x,\vec a)$ is a term. Since $G_{\vec \lambda}(x)$ defines a group, a similar computation to the one done in Lemma \ref{1basedgroupscase1} gives the desired result.
\end{proof}

\begin{exam}
(Ordered vector spaces, example \ref{orderedvectorspace} revisited)

In this example we deal with the base structure $\mathcal{V}=(V, +,0,<,\{\lambda_r\}_{r\in \mathbb{F}})$ and we assume that $\mathbb{F}$ is an ordered field and $V$ is an ordered vector space over $\mathbb{F}$. The theory of 
$\mathcal{V}$ is dense o-minimal, so it is geometric.

As in the previous example, we can consider the case where $R$ and $\FF$ are ordered fields and $[\FF:R]=n$. Then the expansion is strongly dependent and has $dp-rk(\mathcal{V},G)=n+1$. This example can also be studied from the perspective of \cite{BHK}; in particular by Theorem 4.8 \cite{BHK} it was already known that this expansion is NIP. The authors of \cite{BHK} also show this expansion has open core and characterize when it is decidable. 
\end{exam}

\section{Preservation of tree properties: NTP2, NTP1 and NSOP1}\label{sec:preservationTP}

In this section, we prove the expansion preserves other nice properties such as $NTP_2$, NTP$_1$ and NSOP$_1$.
To prove the preservation of $NTP_2$ and NTP$_1$, we will follow the approach from \cite{BeKi}, \cite{DoKi}, but will need to modify several parts of the arguments. The main difference is that in our setting the definable subsets in $G$ involve not only the induced structure from $V$ (as is the case in \cite{BeKi, DoKi}) but also the structure that it carries as an $R$-module. We will follow a similar approach to show the preservation of NSOP$_1$ using as a guide the ideas from \cite{Ra}. Let us start by recalling the basic definitions for NTP$_2$.

\begin{defn}
A theory $T$ has \emph{$k$-TP$_2$} (for some integer $k\geq 2$) if there exist a formula $\varphi (\vec x, \vec y)$ and a set of tuples $\{ \vec a_{i, j} \mid i, j<\og \}$ (in some model of $T$) such that $\{ \varphi (\vec x, \vec a_{i, f(i)})  \mid i<\og \}$ is consistent for every function $f\colon \og \to \og$ and $\{ \varphi (\vec x, \vec a_{i, j}) \mid j<\og \}$ is $k$-inconsistent  for every $i<\og$.
A theory has TP$_2$ if it has $2$-TP$_2$ and a theory has NTP$_2$ if it does not have TP$_2$.
\end{defn}

\begin{defn}
Let $M$ be a structure in a language $\mathcal{L}$. A set of parameters $\{ \vec a_{\mu} \mid \mu\in \og \times \og \}$ in $M$ is called an \emph{indiscernible array} if the $\mathcal{L}$-type of any finite tuple $(\vec a_{\mu_1}, \cdots, \vec a_{\mu_n})$ is determined by the quantifier-free array-type of  the tuple  $(\mu_1, \cdots, \mu_n)$.
\end{defn}

Just as sequences can be enlarged in order to extract indiscernible sequences, for arrays we have
the following result:

\begin{fact}\label{tp2indisc}
If a formula $\varphi(\vec x, \vec y)$ witnesses $k$-TP$_2$ then it may do so with  an indiscernible array. Moreover, for such any such indiscernible array and any function $f\colon \og \to \og$, the collection of formulas $\{ \varphi (\vec x, \vec a_{i, f(i)})  \mid i<\og \}$ has infinitely many realizations.
\end{fact}

Using a result of Chernikov \cite{Ch} reducing the property $TP_2$ to formulas of the form $\varphi(x, \vec y)$ (where $x$ is a single variable) and the fact that we can witness $TP_2$ with indiscernible arrays, we can reduce the problem to:

\begin{prop}\label{step1prop}
A theory has TP$_2$ if and only if there exist a formula $\varphi(x, \vec y)$ (where $x$ is a single variable) and an indiscernible array $\{ \vec a_{i, j} \mid i, j<\og \}$ such that
\begin{enumerate}
\item $\bigwedge_{i<\og} \varphi(x, \vec a_{i, 0})$ has infinitely many realizations,

\item $\bigwedge_{j<\og} \varphi(x, \vec a_{0, j})$ has at most finitely many realizations.
\end{enumerate}
\end{prop}

\begin{proof}
For details, the reader can see \cite[Section 3]{BeKi}.
\end{proof}

Now let us consider the problem in the setting of this paper. As before, we write $\mathcal{L}$ for the language of the vector space (maybe with extra structure), $T$ for its theory and $\mathcal{L}_G$ and $T^{G+}$ are the language and the theory of the associated $G$-structure. As mentioned before, we want to show that $T$ has $NTP_2$ if and only if $T^{G+}$ does. Our first result in this direction deals with the induced structure on the predicate. 

\begin{prop}\label{step2prop}
Assume there exists some $\mathcal{L}_G$-formula $\varphi(x, \vec{y})$ (where $x$ is a single variable) such that   $\varphi(x, \vec{y})\wedge G(x)$ witnesses $k$-TP$_2$ for some $k\geq 2$. Then $T$ has TP$_2$.
\end{prop}

\begin{proof}
Assume that there exists such an $\mathcal{L}_G$-formula $\varphi(x, \vec y)$. By Proposition \ref{step1prop} we may assume that  $\varphi(x, \vec y)\wedge G(x)$ witnesses TP$_2$ with some indiscernible array $\mathcal{A}:=\{ \vec a_{i, j} \ \ \ i, j<\og \}$ and that for every function $f\colon \og \to \og$, the collection $\{ \varphi (\vec x, \vec a_{i, f(i)})  \mid i<\og \}$ is infinite. Furthermore, enlarging the indiscernible array if necessary, we may assume that each $a_{i, j}$ is $G$-independent (for details see \cite[Section 4]{BeKi}). Then, by 
Corollary \ref{inducedstructG}, there exists some $\mathcal{L}$-formula $\psi(x, \vec y)$ and a $\mathcal{L}_{R-mod}$-formula 
$\theta(x,\vec y)$ such that for all $i, j<\og$,
\[  \varphi(x, \vec a_{i, j})\wedge G(x)\, \leftrightarrow \psi(x, \vec a_{i, j})\wedge \theta(x, \vec a_{i, j})\wedge G(x)
\]

Since $\bigwedge_{i<\og} \varphi(x, \vec a_{i, 0})\wedge G(x)$  has infinitely many realizations, 
the conjunction $\bigwedge_{i<\og} \theta(x, \vec a_{i, 0})$ also has infinitely many realizations in $G$. 

\textbf{Claim 1.} $\bigwedge_{j<\og}\theta(x, \vec a_{0, j})$ has at infinitely many realizations. 

Otherwise it has finitely many realizations and by Proposition \ref{step1prop} the $R$-module formula $\theta(x,\vec y)$ has 
TP$_2$. Since the theory of $R$-modules is stable, we obtain a contradiction.

Let us now analyze the formula $\psi(x,\vec y)$. As before, it is easy to see that the conjunction 
$\bigwedge_{i<\og} \psi(x, \vec a_{i, 0})$ also has infinitely many realizations. 

\textbf{Claim 2.} $\bigwedge_{j<\og}\psi(x, \vec a_{0, j})$ has finitely many realizations. 

Otherwise, it has infinitely many solutions.

By Proposition \ref{prop_shift} and Claim 1 we have that
$\bigwedge_{j<\omega}\psi(x, \vec a_{0, j})\wedge \bigwedge_{j<\og}\theta(x, \vec a_{0, j})$ also has infinitely many solutions, a contradiction.

Since the conjunction 
$\bigwedge_{i<\og} \psi(x, \vec a_{i, 0})$ has infinitely many realizations. 
and the conjunction $\bigwedge_{j<\og}\psi(x, \vec a_{0, j})$ has finitely many realizations, by Proposition \ref{step1prop} we conclude that $T$ has TP$_2$.
\end{proof}

As in \cite[Section 4]{BeKi} we can extend the previous result to several variables:

\begin{cor}
 If there exists some $\mathcal{L}_G$-formula $\varphi(\vec x, \vec y)$ such that $\varphi(\vec x, \vec y) \wedge G(\vec x)$
witnesses $k-TP_2$ for some $k \geq 2$ then $T$ has $TP_2$.
\end{cor}

\begin{proof}
It follows from submultiplicity of burden, for details see \cite{BeKi}
\end{proof}

\begin{defn} 
Let $(V,G)\models T^G$ be sufficiently saturated, let $X\subset V^n$ and let $A\subset V$ be a set of parameters. We say that $X$ is $A$-\emph{small} if $X\subset \acl(AG)$. If $X$ is $A$-small for some $A$, then we say $X$ is small, otherwise we say the set $X$ is \emph{large}. Similarly,
for $b\in V$, we write $b\in \scl(A)$ if $b\in \acl(AG)$ and we say that $b$ belongs to the \emph{small closure} of $A$.

\end{defn}

In this section we need to approximate large sets by $\mathcal{L}$-definable sets:
 \begin{prop}\label{uptosmall}
Let $(V,G)\models T^G$ and let $Y\subset V$ be
$\cL_G$-definable. Then there is $X\subset V$ an $\cL$-definable set such
that $Y\triangle X$ is small.
\end{prop}

\begin{proof} 
By Theorem \ref{thm_QE} we may assume $Y=\cup_{j\leq n} Y_j$ where each $Y_j$ is the set of realizations of a formula of the form
 \[ \psi_j( x)\wedge P_{\theta_j} (x) \wedge \bigwedge_k t_k( x)\in G_{\vec \lambda_j}\wedge \bigwedge_l t_l(x)\notin G_{\vec \mu_j} \wedge \bigwedge_k t_k(x)=0\wedge \bigwedge_l t_l(x)\neq 0 \]
 Assume for each $Y_j$ we can find a $\cL$-definable set $X_j$ such that $Y_j\triangle X_j$ is small. Then $\cup_{j\leq n}Y_j\triangle \cup_{j\leq n} X_i\subset \cup_{j\leq n}(Y_j\triangle X_j)$ is small. So it suffice to prove the result for the sets $Y_j$. If the set $Y_j$ is small we can choose $X=\emptyset$. Otherwise we can choose $X_j=\psi_j(V)$.

\end{proof}

Note that the above proposition also follows from the fact that $T^G$ has TEA, see Corollary \ref{TEA} and \cite[Proposition 2.6]{BeVaGroups}. 
We are ready to prove the first main result.

\begin{thm}
If $T^{G+}$ has TP$_2$ then so does $T$.
\end{thm}

\begin{proof}
Assume $T^{G+}$ has TP$_2$. So there exists some $\mathcal{L}_G$-formula $\varphi(x, \vec{y})$ (where $x$ is a single variable) witnessing TP$_2$ with some indiscernible array $\cA:= \{ \vec{a}_{i, j} \mid i, j<\og \}$ and we may assume each element $\vec {a}_{i,j}$ is $G$-independent. There are two possible cases:

\underline{Case 1}. $\bigwedge_{i<\og}\varphi(x, \vec{a}_{i, 0})$ is realized by some $b\in \scl(\cA)$.

Such $b$ is then in the algebraic closure (in the language $\mathcal{L}$) of some tuples $\vec c=(c_1,\dots,c_n)\in \cA$ and $\vec h=(h_1,\dots,h_k)\in G(M)$. Since in the old language the algebraic closure coincides with the vector space span, we have $b\in \dcl(\vec c,\vec h)$ and there are some coefficients in $\mathbb{F}$ such that $b=\sum_{i=1}^n \lambda_{i} c_i+\sum_{j=1}^k \lambda_{j}'h_j$. 
Let $\hat \phi(z_1,\dots,z_k,\vec y;\vec c)$ be the formula 
$$\varphi(\sum_{i=1}^n \lambda_{i} c_i+\sum_{j=1}^k \lambda_{j}'z_j,\vec y)\wedge \bigwedge_{j=1}^kG(z_j)$$

Choose a finite $N$ such that $\vec c$ is part of the sub-array $\{\vec a_{i,j} | i \leq N, j < \omega\}$ and let $\cA':= \{ \vec{a}_{i, j} \mid j<\og, N<i<\omega \}$.
It is then easy to show that the $\mathcal{L}_G$-formula $\hat \phi(z_1,\dots,z_k, \vec{y};\vec c)$ has TP$_2$ with respect to the array $\cA'$ and the result follows from Proposition \ref{step2prop}.

\underline{Case 2}.  All the realizations of $\bigwedge_{i<\og}\varphi(x, \vec a_{i, 0})$ are in $M\setminus \scl(\cA)$.

By Proposition \ref{uptosmall} and  indiscernibility of the array $\cA$, there exists a single $\mathcal{L}$-formula $\psi(x, \vec y)$ such that, for each $i, j<\og$, $\varphi(x, \vec a_{i, j})\triangle \psi(x, \vec a_{i, j})$ defines an $\vec a_{i, j}$-small set. Since the realizations of the conjunction are not small, every realization of $\bigwedge_{i<\og}\varphi(x, \vec a_{i, 0})$ is also a realization of $\bigwedge_{i<\og}\psi(x, \vec a_{i, 0})$. In particular, $\bigwedge_{i<\og}\psi(x, \vec a_{i, 0})$ has infinitely many realizations.

Moreover, $\bigwedge_{j<\og} \psi(x, \vec a_{0, j})$ has only finitely many realizations. (Otherwise, the co-density condition of the predicate $G$ implies that $\psi(x, \vec a_{0,0})\wedge\psi(x, \vec a_{0,1})$ is realized by some $d\in M\setminus \scl( \vec a_{0,0}\vec a_{0,1})$. But then such $d$ also realizes $\varphi(x, \vec a_{0,0})\wedge\varphi(x, \vec a_{0,1})$, contradiction.)

Hence, $T$ has TP$_2$ by Proposition \ref{step1prop}.
\end{proof}

Now we deal with NTP$_1$ (also called NSOP$_2$). We will now follow the strategy from \cite{DoKi}, emphasizing the main differences that are needed to adapt the arguments to the new setting.
We start with an appropriate notion of tree-indiscernability (see \cite[Def. 3.3]{DoKi}) that will play the role of indiscernible array in the argument for NTP$_2$.

\begin{defn} Given ordinals $\alpha,\beta$ we see the set $S=\alpha^{<\beta}$ as a tree. 
We say that a tree $(\vec a_{\eta})_{\eta \in S}$ of compatible tuples of elements of a model $M$ is \emph{strongly indiscernible} over a set $C \subset  M$, if whenever the ordered tree types satisfy $qftp_{tree}(\eta_0,\ldots,
\eta_{n-1}) = qftp_{tree}(\nu_0,\ldots,\nu_{n-1})$ then $tp(a_{\eta_0},\ldots,a_{\eta_{n-1}}/C) = tp(a_{\nu_0},\ldots,a_{\nu_{n-1}}/C)$ for all $n <\omega$ and all tuples $(\eta_0,\ldots,\eta_{n-1})$, $(\nu_0,\ldots,\nu_{n-1})$ of elements of $S$.
\end{defn}

\begin{fact}[Dobrowolski-H.Kim]\label{fact_SOP2DobrowolskiKim}
A theory $T$ has SOP$_2$ if there is a formula $\phi(\vec x,\vec y)$ and a strongly indiscernible
tree $(a_\eta)_{\eta\in 2^{<\omega}}$, such that
\begin{enumerate}
    \item $\{\phi(\vec x,a_{0^n})\mid n<\omega \}$ has infinitely many realizations;
    \item $\phi(\vec x,a_0)\wedge \phi(\vec x,a_1)$ has finitely many realizations.
\end{enumerate}

If $q(\vec x)$ is a type and if instead of condition (1) we have

\begin{enumerate}
\item[(1')] $\{\phi(\vec x,a_{0^n})\mid n<\omega \}\wedge q(\vec x)$ has infinitely many realizations
\end{enumerate}
then we say the theory has $SOP_2$ inside $q(\vec x)$.
\end{fact}

Instead of an argument of burden, in this setting we have.

\begin{fact}[Dobrowolski-H.Kim(Fact 2.7 \cite{DoKi})]\label{fact_SOP2onevariable}
Suppose a theory $T$ has SOP$_2$ inside of some type $q(x_0,..., x_{n-1}) = \bigcup_{i<n} q_i(x_i)$. Then,
for some $i<n$, $T$ has SOP$_2$ inside of $q_i(x_i)$.
\end{fact}

\begin{lem}\label{lm_SOP2G}
Let $\phi(\vec x,\vec y)$ be an $\mathcal L _G$-formula. If $\phi(\vec x,\vec y)\wedge G(\vec x)$ has SOP$_2$, then $T$ has SOP$_2$.
\end{lem}

\begin{proof}
By applying Fact~\ref{fact_SOP2onevariable} with the type $q(\vec x) = G(\vec x)$, we may assume that $|\vec x | =1$. By Fact~\ref{fact_SOP2DobrowolskiKim} there is
a strongly indiscernible tree
$(a_\eta)_{\eta\in 2^{<\omega}}$  witnessing SOP$_1$ and the following properties hold:
\begin{enumerate}
    \item $\{\phi(x,a_{0^n})\wedge G(x)\mid n<\omega \}$ has infinitely many realizations;
    \item $\phi(x,a_0)\wedge \phi(x,a_1)\wedge G(x)$ has finitely many realizations.
\end{enumerate}
By Corollary \ref{inducedstructG}, there exists some $\mathcal{L}$-formula $\psi(x, \vec y)$ and an $R$-module formula 
$\theta(x,\vec y)$ such that for all $i, j<\og$,
\[  \varphi(x, \vec a_{\eta})\wedge G(x)\, \leftrightarrow \psi(x, \vec a_{\eta})\wedge \theta(x, \vec a_{\eta})\wedge G(x)
\]
It is clear that $\bigwedge_n \psi(x,a_{0^n})$ and $\bigwedge_n \theta(x,a_{0^n})$ have infinitely many realisations. As $\theta(x,\vec y)$ is stable, it also follows that $\theta(x,a_0)\wedge \theta(x,a_1)$ has infinitely many realisations (otherwise using Fact~\ref{fact_SOP2DobrowolskiKim} we would have that $\theta(x,\vec y)$ witnesses SOP$_2$).

\textbf{Claim}. $\psi(x,a_0)\wedge \psi(x,a_1)$ has finitely many realisations. \\
Assume not, then $\psi(x,a_0)\wedge \psi(x,a_1)$ has infinitely many realisations and by Proposition \ref{prop_shift}, so does $\psi(x,a_0)\wedge \psi(x,a_1)\wedge \theta(x,a_0)\wedge \theta(x,a_1)\wedge G(x)$, a contradiction.

Since $\psi(x,\vec y)$ is an $\mathcal L$-formula, we conclude that $T$ has SOP$_2$.
\end{proof}

\begin{thm}
If $T^{G+}$ has SOP$_2$ then so does $T$.
\end{thm}

\begin{proof}
Assume $T^{G+}$ has SOP$_2$. By Fact~\ref{fact_SOP2onevariable}, there is a formula $\varphi( x, \vec{y})$ and a strongly indiscernible tree $S = (a_\eta)_{\eta\in 2^{<\omega}}$ that witnesses it. We may also assume each element $\vec {a}_{\eta}$ of the tree is $G$-independent. There are two possible cases:

\underline{Case 1}. $\bigwedge_{i<\og}\varphi( x, \vec{a}_{0^i})$ is realized by some $b\in \scl(S)$.

Such $ b$ is in the $\mathcal{L}$-definable closure of some tuples $\vec c$ in $S$ and $\vec h$ in $G(M)$, so there are coefficients $(\lambda_i)_i, (\rho_j)_j$ in $\FF$ such that $b=\sum \lambda_i h_i+\sum \rho_j c_j$ for some elements $\{s_j\}_j$ in $S$. Let $n<\omega$ be such that $\vec c$ belongs to the subtree
$(a_\eta )_{\eta\in 2^{<\omega}, |\eta|\leq n}$ and choose $\nu\in 2^{<\omega}$ such that $|\nu|>n$. Let $\psi(x, \vec x', \vec y, \vec z)$ be the formula $\varphi(x,\vec y)\wedge  x =
\sum \lambda_i x_i'+\sum \rho_j z_j$. It is then easy to check that the $\mathcal{L}_G$-formula $\exists x\psi(x,\vec x',\vec y,\vec c)\wedge G(\vec x')$ has SOP$_2$ with respect to the strongly indiscernible tree $(a_\eta c)_{\eta\in 2^{<\omega}, \eta > \nu}$, and the result follows from
Lemma~\ref{lm_SOP2G}.

\underline{Case 2}.  All the realizations of $\bigwedge_{i<\og}\varphi(x, \vec a_{ 0^i})$ are in $M\setminus \scl(S)$.

By Proposition \ref{uptosmall}, there exists some $\mathcal{L}$-formula $\psi(x, \vec y)$ such that, for each $\eta\in 2^{<\omega}$, $\varphi(x, \vec a_{\eta})\triangle \psi( x, \vec a_{\eta})$ defines an $\vec a_{\eta}$-small set. Since all realizations are not small over the tree $S$, this implies that every realization of $\bigwedge_{i<\og}\varphi(x, \vec a_{ 0^i})$ is also a realization of $\bigwedge_{i<\og}\psi(x, \vec a_{ 0^i})$. In particular, $\bigwedge_{i<\og}\psi(x, \vec a_{ 0^i})$ has infinitely many realizations.

We show that $ \psi(x, \vec a_{0})\wedge \psi(x, \vec a_{1})$ has only finitely many realizations. Assume not, then by the codensity property for $G$, there some $d\in M\setminus \scl( \vec a_{0}\vec a_{1})$ realizing $\psi(x, \vec a_{0})\wedge\psi(x, \vec a_{1})$. As $\varphi(x,\vec a_0)\Delta \psi(x,a_0)$ is $a_0$-small, $d$ satisfies $\varphi(x,a_0)$, and similarly, $d$ satisfies $\varphi(x,a_1)$. It follows that $d$ satisfies $\varphi(x,a_0)\wedge \varphi(x,a_1)$. Repeating this process and using again the codensity property, there is an infinite sequence $(\vec d_i:i< \omega)$ of realisations of $\psi(x, \vec a_{0})\wedge\psi(x, \vec a_{1})$ such that $\vec d_i \not \in \scl(a_0 a_1\vec d_{<i})$. Hence $\varphi(x,a_0)\wedge \varphi(x,a_1)$ has infinitely many realisation, a contradiction. 

It follows that the $\mathcal{L}$-formula $\psi(x, \vec y)$ witnesses that $T$ has SOP$_2$.
\end{proof}

In particular, since simple theories are those that have at the same time NTP$_2$ and NSOP$_2$ we get the following Corollary.

\begin{cor}\label{preservesimple}
Assume $T$ is simple, then so is $T^{G+}$.
\end{cor}

Now we deal with NSOP$_1$. We will use the following characterization presented in the work by Ramsey \cite{Ra}:

\begin{defn} Let $T$ be a complete theory and let $M\models T$ be sufficiently saturated. We say $T$ has SOP$_1$ if there is a formula $\varphi(\vec x;\vec y)$, possibly with parameters in a set $C$, and array
$(\vec c_{i,j})_{i<\omega,j<2}$ so that
\begin{enumerate}
    \item[(a)] $\vec c_{i,0} \equiv_{C\vec c_{<i}} \vec c_{i,1}$ for all $i < \omega$.
\item[(b)] $\{\varphi(x; \vec c_{i,0}) : i < \omega\}$ is consistent.
\item[(c)] $\{\varphi(x; \vec c_{i,1}) : i < \omega\}$ is 2-inconsistent.
\end{enumerate}
\end{defn}

By Lemma 2.5 and Theorem 2.7 \cite{Ra} it suffices to check the definition above for formulas $\varphi(x;\vec y)$ where $x$ is a single variable and we may choose an array $(\vec c_{i,j})_{i<\omega,j<2}$ that also satisfies

\begin{enumerate}
    \item[(d)] $(\vec c_i)_{i<\omega}$ is an indiscernible sequence.
\item[(e)] $(\vec c_{k,0})_{k\geq i}$ is $\vec c_{<i}c_{i,1}$-indiscernible.
\end{enumerate}
Whenever $(\vec c_{i,j})_{i<\omega,j<2}$ satisfies these extra conditions, we say the array
$(\vec c_{i,j})_{i<\omega,j<2}$ is an \emph{indiscernible array}.

We will need the following modified version of the property that is more suitable to dense pairs. The next result is due to Nicholas Ramsey: 

\begin{prop}[Ramsey]\label{nsop1uptofinite}
A theory $T$ has SOP$_1$ if and only if there exists some set $C$, a formula $\varphi(x, \vec{y})$ (where $x$ is a single variable) and an indiscernible array $(\vec c_{i,j})_{i<\omega,j<2}$ such that
\begin{enumerate}
    \item[(a)] $\vec c_{i,0} \equiv_{C\vec c<i} \vec c_{i,1}$ for all $i < \omega$;
\item[(b)] $\{\varphi(x; \vec c_{i,0}) : i < \omega\}$ is consistent and non-algebraic;
\item[(c)] Whenever $i<j$, $\varphi(x; \vec c_{i,1})\wedge \varphi(x; \vec c_{j,1})$ is finite.
\end{enumerate}
\end{prop}

\begin{proof}
If $T$ has SOP$_1$, then by \cite[Theorem 2.7]{Ra}, there is a formula in one variable witnessing SOP$_1$, in particular, it satisfies conditions $(a),(b)$ and $(c)$. Conversely assume that $C, (\vec c_{i,j})_{i<\omega, j<2}$ and $\phi(x, \vec y)$ are given. Since the array is indiscernible, there is an $N$ such that for any $i<j$ the set 
$\varphi(x; \vec c_{i,1})\wedge \varphi(x; \vec c_{j,1})$ has cardinality $N$. Consider the formula 
\[\psi(\vec x,\vec y) = \psi(x_0,...,x_N,\vec y) = \bigwedge_{0\leq i\leq N}\phi(x_i,\vec y)\wedge \bigwedge_{0\leq i<j\leq N} x_i\neq x_j\]
Using $(b)$, $\set{\psi(\vec x, \vec c_{i,0}) : i<\omega}$ is still consistent. However if $\psi(\vec x, \vec c_{i,1})\wedge \psi(\vec x,\vec c_{j,1})$ were consistent, there would be more than $N$ different realisations of $\varphi(x, \vec c_{i,1})\wedge \varphi(x,\vec c_{j,1})$, so $\psi(\vec x, \vec c_{i,1})\wedge \psi(\vec x,\vec c_{j,1})$ is inconsistent, and $T$ has SOP$_1$.
\end{proof}

We follow the same strategy that as we did earlier in this section, first we check what happens for formulas
in $G(x)$ and then we extend the results to the general case.

\begin{prop}\label{prop_nsop1reductiontoG}
Assume there exists some $\mathcal{L}_G$-formula $\varphi(x, \vec{y})$ (where $x$ is a single variable) such that  $\varphi(x, \vec{y})\wedge G(x)$ witnesses SOP$_1$. Then $T$ has SOP$_1$.
\end{prop}

\begin{proof}
Assume that there exists such an $\mathcal{L}_G$-formula $\varphi(x, \vec y)$ that witnesses SOP$_1$. We may assume that the sequence $(\vec c_i)_{i<\omega}$  witnessesing SOP$_1$ is an indiscernible array and enlarging the tuples if necessary we may assume that each $\vec c_{i, j}$ is $G$-independent. 
Using Corollary \ref{inducedstructG}, there exists some $\mathcal{L}$-formula $\psi(x, \vec y)$ and a $\mathcal{L}_{R-mod}$-formula 
$\theta(x,\vec y)$ such that for all $i, j<\og$,
\[  \varphi(x, \vec a_{i, j})\wedge G(x)\, \leftrightarrow \psi(x, \vec a_{i, j})\wedge \theta(x, \vec a_{i, j})\wedge G(x)
\]

Since $\bigwedge_{i<\og} \varphi(x, \vec c_{i, 0})\wedge G(x)$  is consistent and non-algebraic, clearly the collection of $\mathcal{L}$-formulas $\{ \psi(x, \vec c_{i, 0}): i<\omega\}$ is consistent and non-algebraic. Similarly, $\{ \theta(x, \vec c_{i, 0}): i<\omega\}$ is consistent and non-algebraic.
Also $\vec c_{i,0} \equiv^{\cL}_{C\vec c <i} \vec c_{i,1}$ for all $i < \omega$. We just need to show the almost $2$-inconsistency of the family
$\{\psi(x; \vec c_{i,1}) : i < \omega\}$.

\textbf{Claim 1.} $\theta(x, \vec c_{i, 1})\wedge \theta(x, \vec c_{j, 1})$ has infinitely many realizations whenever
$i< j$.

Otherwise whenever $i< j$ we have that $\theta(x, \vec c_{i, 1})\wedge \theta(x, \vec c_{j, 1})$ has finitely many realizations and by Proposition \ref{nsop1uptofinite} the $R$-module formula $\theta(x,\vec y)$ has 
SOP$_1$. Since the theory of $R$-modules is stable, we obtain a contradiction.

Let us now analyze the formula $\psi(x,\vec y)$. 

\textbf{Claim 2.} $\psi(x, \vec c_{i, 0})\wedge \psi(x, \vec c_{j, 0})$ has  finitely many realizations whenever
$i<j$. 

 Assume not, then $\psi(x, \vec c_{i, 0})\wedge \psi(x, \vec c_{j, 0})$ has infinitely many realisations and by Proposition \ref{prop_shift}, so does $\psi(x, \vec c_{i, 0})\wedge \psi(x, \vec c_{j, 0})\wedge \theta(x, \vec c_{i, 0})\wedge \theta(x, \vec c_{j, 0})\wedge G(x)$, a contradiction.

Now apply again Proposition \ref{nsop1uptofinite} to the formula $\psi(x,\vec y)$ to get the desired result.
\end{proof}

\begin{cor}\label{prop_nsop1tupleinG}
Assume there exists some $\mathcal{L}_G$-formula $\varphi(\vec x, \vec{y})$ (where $\vec x$ is a tuple) such that  $\varphi(\vec x; \vec{y})\wedge G(\vec x)$ witnesses SOP$_1$. Then $T$ has SOP$_1$.
\end{cor}

\begin{proof}
An easy modification of the proof of Theorem 2.7 in \cite{Ra} shows that for some single variable $x_i$ in the tuple $\vec x$ and an appropriate tuple $\vec b$, the formula $\varphi( x_i,\vec b; \vec{y})\wedge G(x_i,\vec b)$ has SOP$_1$. Now apply Proposition \ref{prop_nsop1reductiontoG} to get the desired result.
\end{proof}

With the previous results we are ready to show the last theorem of this section:

\begin{theo}
If $T^{G+}$ has SOP$_1$ then so does $T$.
\end{theo}
\begin{proof}
Assume that $T^{G+}$ has SOP$_1$. Let $C$, $\varphi(x, \vec{y})$ and $S = (\vec c_{i,j})_{i<\omega,j<2}$ be as in Proposition \ref{nsop1uptofinite}, in particular $\bigwedge_{i<\omega} \varphi(x; \vec c_{i,0})$ is consistent and non-algebraic. We may assume that $(\vec c_i)_{i<\omega}$ is an indiscernible sequence and enlarging the tuples if necessary we may assume that each $\vec c_{i, j}$ is $G$-independent. Also, to simplify the presentation we will assume that $C=\emptyset$.\\
\underline{Case 1.} $\bigwedge_{i<\omega} \phi(x; \vec c_{i,0})$ is realised by some $b\in \scl(S)$.\\

Let $\vec d\in S$ and let $ \vec g\in G$ be tuples such that $b = \sum_i \lambda_i d_i+\sum_j \rho_j g_j$ for some coefficients $\{\lambda_i\}_i$, $\{\rho_j\}_j$ in $\FF$. Let $\psi(x, \vec x ', \vec y, \vec z)$ be the formula $\phi(x,\vec y)\wedge x = \sum_i \lambda_i z_i +\sum_j \rho_j x'_j$. Let $n_0<\omega$ be such that $\vec d\subset \set{\vec c_{i,j} : i< n_0,j<2}$. Then for all $n_0\leq i<\omega$, $\vec c_{i,0}\equiv_{C\vec d \vec c<i} \vec c_{i,1}$. It is then easy to see that the formula $\exists x\psi(x, \vec x',\vec y, \vec d)\wedge G(\vec x')$ witnesses SOP$_1$, hence by Corollary \ref{prop_nsop1tupleinG}, $T$ has SOP$_1$.\\
\underline{Case 2.} All realisations of $\bigwedge_{i<\og} \phi(x,\vec c_{i,0})$ are not in $\scl(S)$.\\
By Proposition \ref{uptosmall}, there exists some $\mathcal{L}$-formula $\psi(x, \vec y)$ such that, for each $i<\omega$ and $j<2$, $\phi(x, \vec c_{i,j})\triangle \psi( x, \vec c_{i,j})$ defines a $\vec c_{i,j}$-small set. Since all realizations are not small over $S$, this implies that every realization of $\bigwedge_{i<\og}\phi(x, \vec c_{i,0})$ is also a realization of $\bigwedge_{i<\og}\psi(x, \vec c_{i,0})$. In particular, $\bigwedge_{i<\og}\psi(x, \vec c_{i,0})$ has infinitely many realizations.

We show that $ \psi(x, \vec c_{i,1})\wedge \psi(x, \vec c_{j,1})$ has only finitely many realizations, for $i\neq j$. Assume not, then by the codensity property for $G$, there is some $d\in M\setminus \scl( \vec c_{i,1}\vec c_{j,1})$ realizing $\psi(x, \vec c_{i,1})\wedge\psi(x, \vec c_{j,1})$. As $\phi(x,\vec c_{i,1})\Delta \psi(x,\vec c_{j,1})$ is $\vec c_i$-small, $d$ satisfies $\phi(x,\vec c_{i,1})$, and similarly, $d$ satisfies $\phi(x,\vec c_{j,1})$. It follows that $d$ satisfies $\phi(x,\vec c_{i,1})\wedge \phi(x,\vec c_{j,1})$. Repeating this process and using again the codensity property, there is an infinite sequence $( d_i:i< \omega)$ of realisations of $\psi(x, \vec c_{i,1})\wedge\psi(x, \vec c_{j,1})$ such that $d_i \not \in \scl(c_{i,1} c_{j,1} \vec d_{<i})$. Hence $\phi(x,c_{i,1})\wedge \phi(x,c_{j,1})$ has infinitely many realisation, a contradiction. It follows that the $\mathcal{L}$-formula $\psi(x, \vec y)$ witnesses that $T$ has SOP$_1$.
\end{proof}

\textbf{Question.}
It may be interesting to check in this setting if other model theoretic properties are preserved. For example, is $n$-dependence preserved? Is NSOP$_n$ preserved for $n\geq 3$?

We end this section with some examples.

\begin{exam}\label{genericpredicate}

Consider structures of the form $\mathcal{V}=(V, +,0,S,\{\lambda_r\}_{r\in \FF})$,
where $V$ is a pure vector space over a field $\FF$ and $S$ is a generic subset of $V$ in the sense of Chatzidakis-Pillay \cite{ChPi}. To clarify the notation, set $\mathcal{L}_0=\{+,0,\{\lambda_r\}_{r\in \FF}\}$ and $\mathcal{L}=\mathcal{L}_0 \cup\{S\}$. Since for pure vector spaces over $\FF$ we have that $\acl_{\mathcal{L}_0}=\dcl_{\mathcal{L}_0}=\spane_{\FF}$, by \cite{ChPi} the structure $\mathcal{V}$ has quantifier elimination and $\acl=\dcl=\spane_{\FF}$. The completions of the theory depend on the truth value of $S(0)$ and to simplify the example, assume that $0\in S$. By \cite{ChPi} this structure is simple unstable
of $SU$-rk one and thus geometric.

 Regardless how we choose the submodule $R$,  by Corollary \ref{preservesimple} the theory of the pair $(\mathcal{V},G)$ will always be simple, but as before $\kappa_T$ will depend on how we choose $\FF$ and $R$.
 
 Assume first that $R=\mathbb{Z}$ and $\FF=\mathbb{Q}$. Then the theory of the pair $(\mathcal{V},G)$ will be simple and not supersimple. Assume now that $R=\mathbb{Q}=\FF$, then the expansion $T^G$ corresponds to the theory of a lovely pair of models of $T$ (see \cite{Va}), which is known to be supersimple of $SU$-rank two (see \cite{Va}).
 \end{exam}

For our next example, we will need the following lemmas and proposition showing preservation of $NTP_2$ under generic predicate expansions, assuming $\acl$ is given by linear span. Note that a more general version of this result appears as the Theorem 7.3 in \cite{Ch}; however the argument written in in \cite{Ch} seems incomplete (in using singletons rather than tuples), and at this point we are not aware if it had been fixed.

\begin{lem}\label{step1genpred}
Let $T$ be a complete geometric theory extending the theory of vector spaces over $\FF$ where $acl=\spane_{\FF}$, let $P$ be a new predicate and let $T_P$ be the theory $T$ together with the scheme saying that $P$ is a generic predicate. Let $(M,P)\models T_P$. Then
\begin{enumerate}
    \item For any $\vec a,\vec b\in M^n$, we have $\tp_P(\vec a)=\tp_P(\vec b)$ iff
$\tp(\vec a)=\tp(\vec b)$ and for every $\lambda_1,\dots,\lambda_n\in \FF$,
$P(\lambda_1a_1+\dots+\lambda_na_n)$ holds iff
$P(\lambda_1b_1+\dots+\lambda_nb_n)$ holds.
\item Every $\mathcal{L}_P$-formula $\varphi(x, \vec y)$ is equivalent module $T_P$ to a disjunction of formulas of the form $\psi(x, \vec y)\wedge \theta(x,\vec y)$, where $\psi(x, \vec y)$ is a $\mathcal{L}$-formula and $\theta(x,\vec y)$ is a conjunction of formulas of the form $\pm P(t(x,\vec y))$ and each $t(x,\vec y)$ is an $\FF$-linear combination of the elements $x$, $\vec y$
\end{enumerate}
\end{lem}

\begin{proof}
(1) The left to right direction is clear.
The right direction to left direction follows the argument in Corollary 2.6(b) in  \cite{ChPi} together with the fact that by assumption, whenever $\vec c\in M$, we have $\acl(\vec c)=\spane(\vec c)$. 

(2) Follows from part (1) and compactness. One could also prove this result using  Corollary 2.6(d) in \cite{ChPi}. 
\end{proof}

\begin{lem}\label{ntp2-disjunction}
Suppose $\phi(x,\vec y)=\theta_1(x,\vec y)\vee\ldots\vee \theta_m(x,\vec y)$ has TP$_2$. Then so does $\theta_i(x,\vec y)$ for some $i\le m$.
\end{lem}

\begin{proof}
    It suffices to prove the statement for $m=2$. By Proposition \ref{step1prop}, there exists an  indiscernible array $\{ \vec a_{i, j} \ \ \ i<\og,  j<\og \}$ such that $\bigwedge_{i<\og} \varphi(x, \vec a_{i, 0})$ has infinitely many realizations, and $\bigwedge_{j<\og} \varphi(x, \vec a_{0, j})$ has at most finitely many realizations. Clearly, each of
    $\bigwedge_{j<\og} \theta_1(x, \vec a_{0, j})$ and $\bigwedge_{j<\og} \theta_2(x, \vec a_{0, j})$ has at most finitely many realizations. Thus, it remains to show that at least one of $\bigwedge_{i<\og} \theta_1(x, \vec a_{i, 0})$ or $\bigwedge_{i<\og} \theta_2(x, \vec a_{i, 0})$ has infinitely many realizations. 

Suppose both $\bigwedge_{i<\og}\theta_1(x, \vec a_{i,0})$ and $\bigwedge_{i<\og}\theta_2(x, \vec a_{i,0})$ have finitely many realizations. Then for some $n>0$, both $\bigwedge_{i<n}\theta_1(x, \vec a_{i,0})$ and $\bigwedge_{i<n}\theta_2(x, \vec a_{i,0})$ have finitely many realizations. We claim that
$$\bigwedge_{i<2n+1} \theta_1(x, \vec a_{i,0})\vee \theta_2(x, \vec a_{i,0})$$ has finitely many realizations, which is a contradiction. Indeed, this formula is equivalent to $$\bigvee_{\sigma\in \{1,2\}^{2n+1}}\bigwedge_{i<2n+1}
\theta_{\sigma(i)}(x,\vec a_{i,0}).$$ For any $\sigma\in \{1,2\}^{2n+1}$, the formula $$\bigwedge_{i<2n+1}
\theta_{\sigma(i)}(x,\vec a_{i,0})$$ 
 includes a conjunction of the form $$\theta_1(x,\vec a_{i_1,0})\wedge \ldots \wedge (\theta_1(x,\vec a_{i_n,0})$$ or $$\theta_2(x,\vec a_{i_1,0})\wedge \ldots \wedge (\theta_2(x,\vec a_{i_n,0})$$ for some $0\le i_1<\ldots<i_n<2n+1$. By indiscernibility of $(\vec a_{i,0}:i<\og)$, each such conjunction has finitely many realizations.

\end{proof}

\begin{prop}\label{step2genpred}
Let $T$ be a complete geometric theory extending the theory of vector spaces over $\FF$ where $acl=\spane_{\FF}$, let $P$ be a new predicate and let $T_P$ be the theory $T$ together with the scheme saying that $P$ is a generic predicate.   Assume that there is a $\mathcal{L}_P$-formula $\varphi(x, \vec{y})$ (where $x$ is a single variable) such that $\varphi(x, \vec{y})$ witnesses $k$-TP$_2$ for some $k\geq 2$. Then $T$ has TP$_2$.
\end{prop}

\begin{proof}
Assume that there exists such an $\mathcal{L}_P$-formula $\varphi(x, \vec y)$. By Proposition \ref{step1prop} we may assume that  $\varphi(x, \vec y)$ witnesses TP$_2$ with some indiscernible array $\mathcal{A}:=\{ \vec a_{i, j} \ \ \ i<\og,  j<\og \}$. By Lemma \ref{step1genpred} there exist some $\mathcal{L}$-formulas $\psi_s(x, \vec y)$, $1\le s\le m$, and a list of linear combinations $t_1(x,\vec y),\dots, t_\ell (x,\vec y)\in \spane(x,\vec y)$
such that $\varphi(x, \vec y)$ is equivalent to $\vee_{s=1}^m \psi_s(x, \vec y)\wedge \theta_s(x,\vec y)$, where each $\theta_s(x,\vec y)$ is a conjunction of formulas of the form $\pm P(t_i(x,\vec y))$. We split the argument into cases.

\textbf{Case 1.}The conjunction $\bigwedge_{k<\og}(\vee_{s=1}^m \psi_s(x, \vec a_{0, k}))$ has finitely many realizations.

Note that for any $f:\omega \to \omega$ the collection $\{ \vee_{s=1}^m \psi_s (\vec x, \vec a_{i, f(i)})  \mid i<\og \}$ has infinitely many realizations since the collection $\{ \varphi (\vec x, \vec a_{i, f(i)})  \mid i<\og \}$ already has infinitely many realizations. It follows from Proposition \ref{step1prop} that the $\mathcal{L}$-formula $\vee_{s=1}^m \psi_s(x,\vec y)$ has 
TP$_2$ as desired.

\textbf{Case 2.} 
The conjunction $\bigwedge_{k<\og}(\displaystyle{\vee_{s=1}^m\psi_s(x, \vec a_{0, k})})$ has infinitely many realizations.

By Lemma \ref{ntp2-disjunction}, we may assume that $m=1$, i.e., $\phi(x,\vec y)=\psi(x,\vec y)\wedge\theta(x,\vec y)$.

Now, we have $\phi(x,\vec y)=\psi(x,\vec y)\wedge \bigwedge_{k=1}^{\ell} P^{\delta_k}(t_k(x,\vec y))$, where $\delta_k=0$ or $1$. We may assume that $x$ appears non-trivially in each linear combination. We know that $\phi(x,\vec a_{0,i})$ has infinitely many realizations for every $i$. We also know that $\bigwedge_{i<\omega}\psi(x,\vec a_{0,i})$
has infinitely many realizations. Let $b$ be one such realization, non-algebraic over the sequence. By Erd\"{o}s-Rado Theorem, there exists an $\mathcal{L}$-indiscernible sequence $(b'\vec a_i':i<\og)$ such that for any $i_1<\ldots <i_n<\og$ there exist $j_1<\ldots <j_n<\og$ such that $$\tp_{\mathcal{L}}(b'\vec a_{i_1}',\ldots, b'\vec a_{i_n}')=\tp_{\mathcal{L}}(b\vec a_{0,j_1},\ldots ,b\vec a_{0,j_n}).$$
Thus, we may assume that $(\vec a_{0,i}:i<\omega)$ is $\mathcal{L}$-indiscernible over $b$.
Suppose for some $i<j$ and some $k,m\le \ell$ we have $t_k(b,\vec a_{0,i})=t_m(b,\vec a_{0,j})$. Choosing $n>j$, we have $t_m(b,\vec a_{0,j})=t_k(b,\vec a_{0,i})=t_m(b,\vec a_{0,n})=t_k(b,\vec a_{0,j})$. Thus, by $\mathcal{L}$-indiscernibility over $b$, $t_k(b,\vec a_{0,i})=t_m(b,\vec a_{0,i})$ for all $i<\og$, and thus, also $t_k(b,\vec a_{0,i})=t_k(b,\vec a_{0,j})$ for all $i,j<\og$.
Then, we can also assume that for some $k\le \ell$, we have $t_s(b,\vec a_{0,i})=t_s(b,\vec a_{0,j})$ for all $1\le s\le k$, and all $i,j<\omega$,  $t_s(b,\vec a_{0,i})\neq t_r(b,\vec a_{0,i})$ for $s<r \le k$ (and fixed $i$), and for $k<m_1,m_2\le\ell$ and $i,j<\og$, $t_{m_1}(b,\vec a_{0,i})\neq t_{m_2}(b,\vec a_{0,j})$ whenever $(m_1,i)\neq (m_2,j)$.
Since $b$ appears non-trivially in each term, none of these elements is algebraic over the sequence of $\vec a_{0,i}$'s. Now we assign $P$ to these terms in any way, using the genericity of the predicate, thus, making the conjunction $\bigwedge_{i=1}^\og \phi(x,\vec a_{0,i})$ consistent, a contradiction.
\end{proof}

\begin{exam}

We start with the ordered vector space case (Example \ref{orderedvectorspace}) and add a generic predicate $S$. This expansion eliminates the quantifier $\exists^{\infty}$ (see section 2 in \cite{DMSgenexp}) and has the same algebraic closure as the original theory (see \cite{ChPi}), so it agrees with the span of the parameters  over $\FF$. This vector space is NTP$_2$, not simple (because of the order) and not NIP (because of the generic predicate). The expansion by any submodule will be again NTP$_2$. As before, depending on the choice of $R$ and $\FF$ we can get examples of theories that are strong and others that are not strong.

\end{exam}

\section{Independence relations in the expansion}\label{sec:G-basis}

We start this section by analyzing the case when the base theory $T$ is simple of $SU$-rank one. We first show that forking in $T^G$ has local character, and in doing so, we will extract conditions that give a geometric characterization of independence in the expansion.

\begin{lem}\label{tuple-indep}
    Let $M$ be a model of an SU-rank one theory, $(D_i: i\in\omega)$ an indiscernible sequence in $M$ (where $D_i$ are small), $\vec a\in M$ a finite tuple of elements $\acl$-independent over $D_0$. Then there exists $\vec a'\models \tp(\vec a/D_0)$ such that $(D_i:i\in\omega)$ indiscernible over $\vec a'$ and $\vec a'$ is $\acl$-independent over $\bigcup_{i\in\omega}D_i$.
\end{lem}
\begin{proof}
    The proof will be by induction on the length of $\vec a$. Suppose first that $\vec a=b$ is a single element and that $b\not\in\acl(D_0)$. Fix $n\ge 1$, and let $b_1,\ldots, b_n$ be realizations of $\tp(b/D_0)$ which are $\acl$-independent over $D_0$. Let $\vec b=(b_1,\dots,b_n)$. Then $\tp(\vec b/D)$ does not divide over $\emptyset$, and, thus, we can assume that $(D_i:i\in\omega)$ is indiscernible over $\vec b$. Since $n$ was arbitrary, we have that the partial type in variables $x_1, x_2, \ldots$ over $\bigcup_{i\in\omega}D_i$ saying that $x_i$ are distinct realizations of $\tp(b/D_0)$ and $(D_i:i\in\omega)$ is indiscernible over each of $x_i$ is consistent. It follows that there exists a realization $b'$ of $\tp(b/D_0)$ such that $(D_i:i\in\omega)$ is indiscernible over $b'$ and $b'\not\in\acl(\bigcup_{i\in\omega}D_i)$. This finishes the base case. The induction step is done by adding a part of the tuple $\vec a$ to $D_0$ and repeating the above construction.
\end{proof}

\begin{rem}
 Let $(V_0,G)\models T^{G+}$ let $C\subseteq V_0$ be closed and let $\vec a=(a_1,\dots,a_n)\in V_0$.
 We may assume that for some $k\leq n$,
 we have $dim(a_1,\dots,a_k/CG(V)=k$
 and for $i>k$, $a_i\in \spane_{\FF}(a_1,\dots,a_kCG(V))$. Assume $a_i=\sum_{j\leq k} \mu_j a_j+c+
 \sum_{j\leq \ell} \lambda_j g_j$ for some $c\in C$ and $g_j\in G$. Choose $\ell$ minimal with with property, so the tuple $(g_1,\dots,g_{\ell})\in G$ is algebraically independent over $C$. Then $a_i$ is interalgebraic in the extended language (using the functions $f_{\vec \lambda,i}$) with $(g_1,\dots,g_{\ell})$ and thus we may assume, when considering a tuple over $C$, that $\vec a=(\vec a_1,\vec a_2,\vec a_3)$, where $\vec a_1$ is independent over $CG$, $\vec a_2\in G$ is independent over $C\vec a_1$ 
 and $\vec a_3\in dcl(\vec a_1,\vec a_2,C)$.
 \end{rem}

\begin{theo}\label{forking-localchar}
Assume that $T$ has $SU$-rank one. Then dividing in $T^G$ has local character (and, therefore, $T^G$ is simple).
\end{theo}
\begin{proof}
We work inside a sufficiently saturated model  $(\mathcal{V},G)\models T^G$. Let $D \subseteq V$ be small, $D=\langle D\rangle$.

Consider a tuple of the form $\vec a=\vec a_1\vec a_2\vec a_3$, where $\vec a_1$ is a $D\cup G(V)$-independent tuple, $\vec a_2\in G(V)$ is $\FF$-linearly independent over $D$ and $\vec a_3\in \spane_{\mathbb{F}}(\vec a_1 \vec a_2D)$.
Note that under this assumption $\dim(\vec a_1\vec a_2/D)=|\vec a_1\vec a_2|$.
Let $C\subseteq D$ be such that $\tp_R(\vec a_2/G(D))$ does not fork over $G(C)$ in the sense of $R$-modules and such that $\vec a_3\in\spane_\FF(\vec a_1\vec a_2C)$. We will show that $\tp_G(\vec a/D)$ does not fork over $C$.

Let $(D_i
: i \in  \omega)$ be an $\mathcal{L}_G$-indiscernible sequence over $C$ with $D_0=D$.
We start with $\vec a_1$. 

\textbf{Claim 1} There exists $\vec a'_1\models \tp_G(\vec a_1/D)$ such that $\vec a_1'$ is 
$\FF$-linearly independent over $G(V)\cup\bigcup_{i\in\omega} D_i$ and $(D_i:i\in\omega)$ is $\mathcal{L}_G$-indiscernible over $\vec a'_1C$.

Let $p(x, D) =\tp(\vec{a}_1, D)$. 
Since $\vec a_1$ is independent over $D$, $\tp(\vec a_1/D)$ does not divide over $C$ (in fact, over $\emptyset$) and
$\cup_{i\in \omega} p(x, D_i)$ is consistent. By Lemma \ref{tuple-indep} we can find $\vec a'_1$ with $\vec a'_1\models \cup_{i\in \omega} p(x, D_i)$ such that $\{\vec a'_1D_i
: i \in \omega \}$ is indiscernible over $C$ and $\vec a'_1$
is $\FF$-linearly independent over $\cup_{i\in \omega} D_i$.
Applying the \emph{extension property} in each of the components of $\vec a_1'$, we may assume that $\vec a'_1$ is $\FF$-linearly independent over $\cup_{i\in \omega} D_iG(V)$. Note that $G(\spane_{\FF}(\vec a_1D))=G(D)$ and that
$G(\spane_{\FF}(\vec a_1'D_i))=G(D_i)$ for all $i$. 
Thus $\tp_G(\vec a_1 'D_i)=\tp_G(\vec a_1D)$ for all $i$. Applying Erd\"{o}s-Rado theorem, we can assume that $\{\vec a'_1D_i
: i \in \omega \}$ is $\mathcal{L}_G$-indiscernible over $C$. Since the $\mathcal{L}_G$-type of any increasing tuple of the new sequence coincides with that of one of the increasing tuples of the old sequence, we still have that $\vec a'$ is $\FF$-linearly independent over $\cup_{i\in \omega} D_i$.

Now we deal with the part that belongs to $G$. 
We may assume that $\vec a_1=\vec a_1'$ satisfies the statement of Claim 1.

\textbf{Claim 2} 
There exists  $\vec  a_2'$ such that for all 
$i\in\omega$, we have $\tp_G(D_i\vec a_1\vec a_2')=\tp_G(D\vec a_1\vec a_2)$.
\\
Let $q(x, D) =\tp_R(\vec{a}_2, G(D))$. Note that $(G(Di)
: i \in  \omega)$ is an $R$-module indiscernible sequence over $G(C)$. Since $\vec a_2$ is $R$-module independent from $G(D)$ over $G(C)$ we get that
$\cup_{i\in \omega} q(x, G(D_i))$ is consistent. 
\\
Let $p_2(x, D\vec a_1) =\tp(\vec{a}_2/ D\vec a_1)$. By construction, the sequence $(D_i\vec a_1: i \in  \omega)$ is an $\mathcal{L}_G$-indiscernible sequence over $C$, hence, also $\mathcal{L}$-indiscernible over $C$. Since $\vec a_2$ is a independent tuple over $D\vec a_1$, 
$\tp(\vec a_2/D\vec a_1)$ does not $\mathcal{L}$-divide over $\emptyset$, hence, also over $C\vec a_1$ and $\cup_{i\in \omega} p_2(x, D_i\vec a_1)$ is consistent.

We need to make the $R$-module extension $q_\omega(x)=\cup_{i\in \omega} q(x, G(D_i))$ and the $\mathcal{L}$-extension $p_{2\omega}(x)=\cup_{i\in \omega} p_2(x, D_i\vec a_1)$ compatible. Since $\dim(\vec a_2/D\vec a_1)=|\vec a_2|$, 
by Lemma \ref{tuple-indep} we also have $\dim(p_{2\omega}(x))=|\vec a_2|$ and the consistency of $q_\omega(x) \cup p_{2\omega}(x)$ follows from Corollary \ref{cor_shift}. Let $\vec a_2'\in G^{|\vec x_2|}$ be such that $\vec a_2' \models p_{2\omega}(x)\cup q_\omega(x)$. Note that $\vec a_2'$ belongs to $G$ and thus for each $i\in\omega$, $D_i\vec a_1\vec a_2'$ is $G$-independent, $\tp(D_i\vec a_1\vec a_2')=\tp(D\vec a_1\vec a_2)$ and $\tp_G(\vec a_2, G(D))=\tp_G(\vec a_2', G(D_i))$ and thus 
$\tp_G(D_i\vec a_1\vec a_2')=\tp_G(D\vec a_1\vec a_2)$, as needed.

Finally let 
$\vec a_3'$ be such that $\tp_G(D\vec a_1\vec a_2'\vec a_3' )=\tp_G(D\vec a_1\vec a_2\vec a_3)$. Then $\vec a_3'\in\spane_\FF(\vec a_1\vec a_2' C)$, and, therefore,  for all $i\in\omega$, $\tp_G(D_i\vec a_1\vec a_2'\vec a_3')=\tp_G(D\vec a_1\vec a_2\vec a_3)$.
Recall that the requirements on the set $C$ are that $\tp_R(\vec a_2/G(D))$ does not fork over $G(C)$ in the sense of $R$-modules and such that $\vec a_3\in\spane_\FF(\vec a_1\vec a_2C)$, and therefore, there is a uniform bound on the size of $C$. Hence, dividing in 
 $T^G$ has local character, and, thus $T^G$ is simple.
Furthermore, the local character associated to $T^G$ is the local character associated to forking in $R$-modules, since for the condition $\vec a_3\in\spane_\FF(\vec a_1\vec a_2 C)$ it suffices to pick a finite subset $C\subset D$.
\end{proof}

The main goal of this section is to study independence relations in the expansion. We start by extracting the information that we needed for the argument above to work. For Claim 1 and Claim 2,  we require $\dim(\vec a/C)=\dim(\vec a/D)$, that is, algebraic independence should hold. For claim 1 we need $\dim(\vec a/CG(V))=\dim(\vec a/DG(V))$, so we need algebraic independence after localizing in $G$. Finally for claim 2 we need $R$-module independence in the components that belong to $G$. This will give us intuition when we set up Definition \ref{forkingVG} below.
In any case, we point out a strong limitation of the argument above:

\begin{rema}
In the argument above we used the extension property to guarantee 
that $\dim(\vec a_1'/G(V)D)=\dim(\vec a_1'/G(V)C)$. The extension property only guarantees the existence of such tuples for algebraic independence localized in $G(V)$, not for other forms of independence that may properly extend algebraic independence. 
\end{rema}

Now we introduce a more general setting for the rest of this section. Assume the structure $\mathcal{V}=(V,+,0,\{\lambda\}_{\lambda\in \mathbb{F}},\dots)\models T$ is sufficiently saturated and assume the models of its theory have a natural independence relation $\ind^0$ that implies (but may extend properly) \emph{algebraic independence}. We will reserve the symbol $\ind$ for algebraic independence, as we did earlier in the text.

Key to the arguments below is the following result of \cite{DK21} (see also \cite[Fact 4.6]{CdEHJRK22},):

\begin{fact}\label{fact:KPnsop1}
Let $T$ be a first order theory. Then $T$ is NSOP$_1$ if and only if there is an automorphism invariant ternary relation $\ind^*$ on small subsets of a
monster model, only allowing models in the base, satisfying the following properties:
\begin{enumerate}
  \item symmetry. If $A\ind^*_M B$ then $B\ind^*_M A$,
  \item existence. $A\ind^*_M M$, for all $M$.
  \item finite character. If $\vec a\ind^*_M B$ for all finite $\vec a\subseteq A$, then $A\ind^*_M B$.
  \item monotonicity. If  $A\ind^*_M BD$ then $A\ind^*_M B$.
  \item transitivity. For $M\subseteq N$ if $A \ind^*_{ N} B$ and $ N \ind^*_M B$ then $A N \ind^*_M B$.
  \item extension. If $A\ind^*_M B$, then for all $D$ there exists $A'\equiv_{M B}A$ and $A'\ind^*_M BD$.
  \item local character*. Let $\vec a$ be a finite tuple and $\kappa >\abs{T}$ a regular cardinal. For every continuous chain $(M_i)_{i<\kappa}$ of models with $\abs{M_i}<\kappa$ for all $i<\kappa$ and $M=\bigcup_{i<\kappa} M_i$, there is $j<\kappa$ such that $\vec a\ind^*_{M_{j}} M$.
\item independence theorem over $M$. If there exists $C_1,C_2$ and $A,B$ such that $C_1\equiv_M C_2$, $A\ind^*_M B$, $C_1\ind^*_M A$ and $C_2\ind^*_M B$. Then there exists $C$ such that $C\ind^*_M A,B$, $C\equiv_{M A} C_1$ and $C \equiv_{M B} C_2$.
\end{enumerate}
Furthermore, in this case $\ind^*$ is Kim-independence over models.
\end{fact}

 We will \textbf{assume} that $\ind^0$ has the basic properties listed above: invariance, symmetry, existence, finite character, monotonicity, transitivity and extension. We will write $\ind^{R-mod}$ for forking independence inside $R$-modules.
 
 \begin{obse} Let $V_0 \subset V$ be such that $V_0\models T$, let $D\subset V$ be a set and let $\vec a$ be a finite tuple. Then whenever $\vec a\ind^0_{V_0} D$, we also have
$\spane_{\FF}(\vec a)\ind^0_{V_0} \spane_{\FF}(D)$. 

Indeed, by extension, there exists $a'\equiv^{\mathcal{L}}_{V_0D} a$ with $a'\ind^0_{V_0}\spane_{\FF}(D)$. If $\sigma$ is an $\mathcal{L}$-automorphism over $V_0D$ sending $a'$ on $a$, then by invariance of $\ind^0$ we also get $a\ind^0_{V_0} \spane_{\FF}(D)$, and we conclude by the result by symmetry.
 \end{obse}

We now introduce an independence relation in the expansion 
$(\mathcal{V},G)\models T^{G+}$ and assume the expansion is sufficiently saturated.

\begin{defn}\label{forkingVG}
Let $V_0 \subset D\subset V$ be such that $V_0\models T^{G+}$, $D$ is algebraically closed in $T^{G+}$ and let $\vec a\in V^n$. For $\vec \lambda = (\lambda_1,\ldots,\lambda_n)\in \FF^n$, denote $G_{\vec \lambda} = \lambda_1G+\ldots +\lambda_nG$. Define $\vec a\ind^G_{V_0} D$ if and only if 
\begin{enumerate}

\item $\langle V_0\vec a\rangle \ind^0_{V_0} D$

\item $G(\langle \vec a V_0 \rangle) \ind^{R-mod}_{G(V_0)}{G(D)}$. 
 \item $G(\vect{\vec a V_0}+ D) = G(\vect{\vec a V_0})+G(D)$.
\end{enumerate}
\end{defn}

In the above definition above we require the set $D$ to be algebraically closed in $T^{G+}$ and in particular it is $G$-independent. Note that for 
$\vec a$, $\vec d$ tuples, we can define $D'=\langle V_0\vec d \rangle$, it satisfies $V_0\subset D'$, so we can extend the definition and say that $\vec a\ind^G_{V_0} \vec d$ holds whenever we have $\vec a\ind^G_{V_0} D'$. Similarly, for a set $A$ we write $A\ind^G_{V_0} D$ if and only if $\vec a\ind^G_{V_0} D$ for all finite $\vec a\subset A$.
Note that by definition, we have $\vec a\ind^G_{V_0} D$ if and only if $\vec aV_0\ind^G_{V_0} D$.

Let us see how condition (3) above arises naturally in our setting.
A key observation is the following strong connection between forking and cosets of definable groups.

\begin{rem}\label{cosetsforking}
    Let $T$ be any theory of abelian group in which there is a predicate $G$ for a subgroup of the ambient theory. Assume that for any $b\notin G$ there exists an infinite sequence $(b_i)_{i<\omega}$ in $\tp(b/C)$ such that $b_i-b_j\notin G$ for all $i\neq j$.
    Denote by $\ind^d$ the dividing independence relation. Then for any subgroups $C\subseteq A\cap B$
    \[A{\ind}^d_C B\implies G(A+B) = G(A)+G(B)\]
    First, we prove that for any $b$ the formula $\phi(x,b)$ $2$-divides over $C$, where $\phi(x,y)$ is defined by $x+y\in G\wedge y\notin G$. If $b\in G$ the formula $\phi(x,b)$ is inconsistent hence $2$-divides over emptyset. Assume that $b\notin G$ and let $(b_i)_{i<\omega}$ be as in the hypotheses, in $\tp(b/C)$. If $a$ is a realisation of $\phi(a,b_i)\wedge \phi(a,b_j)$, then $a-b_i\in G$ and $a-b_j\in G$ implies $b_i-b_j\in G$ which is a contradiction. So $\set{ \phi(x,b_i)\mid i<\omega}$ is $2$-inconsistent. Now we prove the above implication, or rather its contrapositive. Assume that for $C\subseteq A\cap B$ we have $G(A+B) \neq G(A)+G(B)$ then there exists $a\in A$ and $b\in B$ such that $a+b\in G$ and $a+b\notin G(A)+G(B)$. As $a+b\in G$, we have that \[a\in G \iff b\in G\]
    (using that $A$ and $B$ are groups). In particular, if $b\in G$ then $b\in G(B)$ and $a\in G(A)$ hence $a+b\in G(A)+G(B)$, a contradiction. It follows that neither $a$ nor $b$ is in $G$. In particular $\phi(x,b)\in \tp(A/B)$ hence $\tp(A/B)$ divides over $C$.
\end{rem}

\begin{rem}
    The same argument gives that Kim-dividing implies this condition as long as for any $b\notin G$ there exists an invariant/f.s./non-forking-Morley sequence in $\tp(b/C)$ such that $b_i-b_j\notin G$. A case where such sequences does not exist: consider a theory $T$ of abelian group where there is a unary function symbol $f$ which is a group homomorphism and assume that $G = \ker f$ in this theory. Now assume that $b\notin G$ but that $f(b) = c\in C$. Then for any sequence $(b_i)_{i<\omega}$ in $\tp(b/C)$ we have $f(b_i) = c = f(b)$ hence $f(b_i-b_j) = 0$ i.e. $b_i-b_j\in \ker f = G$. In that case the formula $\phi(x,y) = x+y\in G\wedge y\notin G$ does not necessarily divide over $C$.
\end{rem}

\begin{lem}
The independence relation $\ind^G$ satisfies invariance and monotonicity.
\end{lem}

\begin{proof}
Invariance follows from the fact that $\ind^0$ satisfies invariance, that condition (3) is invariant under automorphisms and the fact that, $\ind^{R-mod}$ being forking independence in $R$-modules, also satisfies invariance.

Let us consider now monotonicity. Given two sets $A\subseteq B\subseteq V$, we have that  $G(\langle A \rangle)\subseteq G(\langle B \rangle)$, since monotonicity holds for forking-independence, in stable theories, we have that item (2) in the definition of $\ind^G$ is preserved after taking subsets of $C$ and $D$. Since $\ind^0$ satisfies monotonicity, it remains to show that monotonicity holds for condition (3), which is easy, and where done in more generality in \cite[Section 3]{De}.

\end{proof}

Now we point out some observations on $\FF$-linear independence.

\begin{lem}
The relation $\ind^G$ is symmetric.
\end{lem}

\begin{proof}
Fix tuples $\vec a,\vec d$. Clearly items (1) and (2) are symmetric conditions since $\ind^0$ satisfies symmetry
as well as forking-independence in modules. Condition $(3)$ satisfies symmetry once we change the sets for their $\mathcal{L}_G$-algebraic closures with $V_0$.
\end{proof}

\begin{lem}
Assume that $\ind^0$ is transitive, then so is $\ind^G$.
\end{lem}

\begin{proof}
Let $(V_0,G)\subseteq (V_1,G)\models T^{G+}$ and assume that $A \ind^G_{V_1} B$ and $V_1 \ind^G_{V_0} B$. Using the definition and the fact that $\ind^0$ is transitive we get $A V_1 \ind^{G}_{V_0} B$. Similarly, since transitivity holds for forking-independence in $R$-modules we also get the condition $G(\langle A V_1\rangle) \ind^{R-mod}_{G(V_0)} G(\langle B V_0\rangle )$ as desired.  
It remains to show condition (3) which is easy, and is done in more generality in \cite[Section 3]{De}

\end{proof}

\begin{lem}
The independence relation $\ind^G$ satisfies extension.
\end{lem}

\begin{proof} Let $(V_0,G)\models T^{G+}$, let $B$, $D$ be sets, and let $p(\vec x)$ be a $\cL_{G+}$-type over $V_0B$ such that whenever $\vec a\models p$ we have $\vec a\ind^G_{V_0}B$. After adding to $\vec a$ elements from $\langle \vec a V_0\rangle$ if necessary, we may assume that if $\vec a\models p$, then we can write $\vec a=(\vec a_1,\vec a_2,\vec a_3)$, where $\vec a_1\ind_{V_0} G(V)$ and it is $\FF$-linearly independent over $V_0$, $\vec a_2\in G(V)$ and $\vec a_2$ is  $\FF$-linearly independent
over $V_0\vec a_1$,  and 
$\vec a_3\in \spane_{\FF}(\vec a_1\vec a_2V_0)$. Since $\ind^0$ satisfies extension 
we can find $\vec a'\models p$ such that
$\vec a'=(\vec a_1',\vec a_2',\vec a_3')\ind^0_{V_0}BD$. 

Since $\dim(\vec a_1/V_0)=|\vec a_1|$ and $\ind^0$ extends algebraic independence, we also have $\dim(\vec a_1'/V_0BD)=\dim(\vec a_1'/V_0)=|\vec a_1|$. By the codensity property in the pair $(V,G)$
we may assume $\dim(\vec a_1'/ G(V)BD)=|\vec a_1'|$, so the tuple $\vec a_1'$ can also be chosen to be $\FF$-linearly independent over $BDG(V)$ $(*)$. 

Now let $p_2(\vec x_2)=\tp_{\cL}(\vec a_2'/V_0BD\vec a_1')$ and let $p_{2}^R(\vec x_2)=\tp_{R-mod}(\vec a_2/G(\langle V_0B\rangle ))$. Since $\ind^{R-mod}$ is a stable independence notion, it satisfies existence and we can find $p_{2D}^R(\vec x_2)$ a non-forking extension of $p_{2}^R(\vec x_2)$ over $G(\langle V_0BD \rangle)$. 

Using Corollary \ref{cor_shift} we may assume that $\vec a_2'\models p_2(\vec x_2)\cup p_{2D}^R(\vec x_2)$. Let $\vec a'=(\vec a_1',\vec a_2',\vec a_3')$.

We check that condition (3) is satisfied, i.e. $G(\vect{\vec a'}+\vect{BD}) = G(\vect{\vec a'})+G(\vect{BD})$. Let $x = x_1+x_2+b+d\in G$ be in the left hand side, for $x_1\in\vect{\vec a_1'}, x_2\in \vect{\vec a_2'},b\in B,d\in D$. Then $x_1\in \vect{\vec a_1'}\cap \vect{G(V)BD} = \set{0}$, by $(*)$. So $x = x_2+b+d\in G$. It is clear that $b+d\in G(\vect{BD})$ so it follows that $x\in G(\vect{\vec a_2'})+G(\vect{BD})$.

This shows that $\vec a'\ind^G_{V_0}D$. Since both tuples $(\vec a_1,\vec a_2)$ , $(\vec a_1',\vec a_2')$ are $G$-independent over $V_0B$, $G(\vec a_1,\vec a_2)=\vec a_2$ and $G(\vec a_1',\vec a_2')=\vec a_2'$ we get by Corollary \ref{G-ind-qe} that $\tp_{G+}(\vec a/V_0B)=\tp_{G+}(\vec a'/V_0B)$ and $\vec a'$ is the desired tuple.
\end{proof}

\begin{rem}\label{rk:locforG}
    If $A,B$ are $\FF$-vector spaces in a model of $T^G$ then there exists $B_0\seq B$ with $\abs{B_0}\leq \abs{A}$ such that $G(A+B) \seq G(A+B_0)+G(B)$. To see this, for each $a\in A$, define $b_a$ to be any element of the set $\set{b\in B\mid a+b\in G}$ if it is nonempty, or else $b_a = \emptyset$. Set $B_0 = \bigcup_{a\in A} b_a$ and take any $g\in G(A+B)$, then $g = a+b = a+b_a +(b-b_a)$ and $b-b_a\in G$ hence $g\in G(A+B_0)+G(B)$.
\end{rem}

\begin{lem}
Assume that $\ind^0$ satisfies local character*, then so does $\ind^G$.
\end{lem}

\begin{proof}

Let $\vec a$ be a finite tuple and let $\kappa >|T^{G+}|=|T|$ be a regular cardinal. Let $\{(V_j,G_j)\}_{j<\kappa}$ be a continuous chain of models of $T^{G+}$ and let $(V,G)=\cup_{j<\kappa}(V_j,G_j)$. Then $\{V_j\}_{j<\kappa}$ is also
a continuous chain and by hypothesis there is $j_1<\kappa$ such that $\vec a\ind^0_{V_{j_1}} V$. 

Since forking independence in $R$-modules is stable, let $\kappa_R$ be the corresponding local character cardinal and note that $\kappa_R\leq |T^{G+}|$. 
Let $B\subset G(V)$ be of cardinality $\leq \kappa_R$ be such that $G(\langle \vec a B \rangle)\ind^{R-mod}_{G(B)}G(V)$. Since $\kappa$ is regular and $\kappa_R<\kappa$, we may find $j_2<\kappa$ such that 
$B\subset G(V_{j_2})$ and thus 
$G(\langle \vec a G(V_{j_2}) \rangle)\ind^{R-mod}_{G(V_{j_2})}G(V)$. 

Let $j_3=\max(j_1,j_2)$, then we have $\vec a\ind^0 _{V_{j_3}} V$
and $G(\langle \vec a G(V_{j_3}) \rangle)\ind^{R-mod}_{G(V_{j_3})}G(V)$.

 Let $\tau$ be the operator defined by setting 
\[\tau(A) = \spane_\FF(A\cup \bigcup_{\vec\lambda,i} f_{\vec\lambda,i}(A)).\] Then $A = \vect{A}$ if and only if $f_{\vec\lambda,i}(A) \seq A$ for all $\lambda$
and all $i$. Using Remark \ref{rk:locforG}, we iteratively construct a sequence $(A_n)_{n<\omega}$ with $A_0 = \vect{\vec a V_{j_3}}$ and such that for some sequence $(V_{k_n})_{n<\omega}$ we have:
\begin{enumerate}
    \item $\tau(A_n)+V_{k_n}\seq A_{n+1}$.
    \item $A_{n}\seq \vect{\vec a V_{k_n}}$.
    \item $j_3\leq k_0\leq k_1\leq \ldots <\kappa$.
    \item $G(A_n+V)\seq G(A_{n+1})+G(V)$.
\end{enumerate}
(Start with $A_0 = \vect{\vec a V_{j_3}}$ and apply Remark \ref{rk:locforG} to $\tau(A_0)$ to get $k_0\geq j_3$ and $V_{k_0}$ such that $G(\tau(A_0)+V)\seq G(\tau(A_0)+V_{k_0})+G(V)$. Then define $A_1 = \tau(A_0)+V_{k_0}\seq \vect{\vec aV_{k_0}}$, we have $G(A_0+V)\seq G(\tau(A_0)+V)\seq G(A_1)+G(V)$. Then apply again Remark \ref{rk:locforG} with $\tau(A_1)$, etc... An iteration of this process yields the desired sequences $(A_n)_{n<\omega}$ and $(V_{k_n})_{n<\omega}$.)

Let $A = \bigcup_{n<\omega} A_n$, so $G(A+V) = G(A)+G(V)$ by (4). By (3), as $\kappa$ is regular, there exists $V_k$ such that $A= \vect{\vec a V_{k}}$. As $k\geq j_3$ we conclude that $\vec a\ind^0_{V_k} V$, $G(\vect{\vec aV_k}) \ind^{R-mod}_{G(V_{k})} G(V)$ and $G(\vect{\vec aV_k}+V) = G(\vect{\vec a V_k})+G(V)$ hence $\vec a \ind^G_{V_k} V$.
\end{proof}

\begin{rem}\label{rk:heir}
    Assume that for some $A,D\seq G$ we have $A\ind^{R-mod}_{G(V_0)} D$. Let $a\in A,d\in D$ and $r\in R$ be such that $a \in  d + rG$ then there exists $v_0\in G(V_0)$ such that $a\in v_0+rG$. 
\end{rem}
\begin{proof}
    The theory of $R$-modules is stable, hence as $\tp(a/DG(V_0))$ (in the sense of $R$-modules) does not fork over the model $G(V_0)$ we get that $\tp(a/DG(V_0))$ is an heir of $\tp(a/G(V_0))$, hence there exists $v_0\in G(V_0)$ such that the formula $\exists y (\in G)  x = v_0+ry$ belongs to $\tp(a/G(V_0))$.
\end{proof}


\begin{theo}
Assume that $T$ has $SU$-rank one and that $\hat{R}=\FF$. Then the independence theorem also holds for $\ind^G$.
\end{theo}

\begin{proof}
Let $(V_0,G)\models T^{G+}$ and let $B,C\supset V_0$ be supersets such that $B\ind^G_{V_0} C$. In particular by clause $(1)$ of $G$-independence we get $B\ind^0_{V_0} C$ and, exchanging the sets for their $\mathcal{L}_G$-algebraic closure, we may assume that $C=\langle C\rangle$ and $D=\langle D\rangle$.

Let $p(\vec x)$ be an $\mathcal{L}_G$-type over $V_0$ and let $p_B(\vec x)\in S(B)$, $p_C(\vec x)\in S(C)$ be independent extensions of $p(\vec x)$ in the sense of $\ind^G$. If  $\vec a\models p(\vec x)$, after adding elements of $\langle \vec a V_0\rangle$ if necessary, we may assume that any realization of $p(\vec x)$ is of the form $\vec a=\vec a_1\vec a_2\vec a_3$, where
$\vec a_1$ is a $V_0\cup G(V)$-independent tuple, $\vec a_2\in G(V)$ is $\FF$-linearly independent over $V_0$ and $\vec a_3\in \dcl_{\mathbb{F}}(\vec a_1 \vec a_2V_0)$. Note that under this assumption $\vec a$ is $G$-independent over $V_0$.

Let $p_{B}^\mathcal{L}(\vec x)$ be the restriction of $p_B(\vec x)$ to the language $\mathcal{L}$. Similarly let $p_{C}^{\mathcal{L}}(\vec x)$ be the restriction of $p_C(\vec x)$ the language $\mathcal{L}$.
Since $B\ind^0_{V_0} C$, we use the independence theorem for $\ind^0$-independence to find $\vec a_1 \vec a_2 \vec a_3\models p_{B}^{\mathcal{L}}(\vec x) \cup p_{C}^{\mathcal{L}}(\vec x)$ such that $\vec a_1 \vec a_2 \vec a_3\ind_{V_0}^0B\cup C$. 

Using the codensity property, we can assume that $\vec a_1$ is $\FF$-linearly independent from $G(V)$ over $B\cup C$. 

We could use the density property and assume that $\vec a_2\in G$, but the module type of $\vec a_2$ over $V_0$ may not be the desired one. Instead, let $p^{\mathcal{L}}_{2BC\vec a_1}(\vec x_2)=\tp_{\cL}(\vec a_2/BC\vec a_1)$; we will refine the way we choose the realization of this last type in order to take into account the module structure.

Let $p_{2B}(\vec x_2)$ be the restriction of $p_B(\vec x)$ to the $|\vec a_2|$-coordinates considered over the parameters set $B$ and let $q_{2B}^{R}(\vec x_2)$ be its restriction to the language of $R$-modules over the parameter set $G(B)$. Similarly let $p_{2C}(\vec x_2)$ be the restriction of $p_C(\vec x_2)\in S(C)$ to the $|\vec a_2|$-coordinates considered only over the parameters set $C$ and let $q_{2C}^{R}(\vec x_2)$ be its restriction to the language of $R$-modules to $G(C)$. Finally let $q_2^{R}(\vec x_2)$ be the restriction of $p(\vec x)$ to the $|\vec a_2|$-coordinates considered only over the parameters set $G(V_0)$ restricted to the language of $R$-modules. 

Since $C\ind^G_{V_0}D$ we get using clause $(2)$ that $G(C)\ind^{R-mod}_{G(V_0)}G(D)$.

Since the theory of $R$-modules is stable, $q_{2C}^{R}(\vec x_2)$ is the unique non-forking extension of $q_2^{R}(\vec x_2)$ to $G(C)$ and $q_{2B}^{R}(\vec x_2)$ is the unique non-forking extension of $q_2^{R}(\vec x_2)$ to $G(B)$ which can be amalgamated to the unique non-forking extension of $q_2^{R}(\vec x_2)$ to $G(B)\cup G(C)$, which we will call $q_{2BC}^{R}(\vec x_2)$.

Now we show that both extensions
$p_{2BC\vec a_1}^{\mathcal{L}}$ and 
$q_{2BC}^{R}$ are compatible and can be extended to a type over $B\cup C$ that does not fork over $V_0$ in the sense of $\ind^G$.

\textbf{Claim 1.} $ q_{2BC}^{R}(\vec x_2)\cup p^{\mathcal{L}}_{2BC\vec a_1}(\vec x_2)$ is consistent and any realization is independent in the sense of $R$-modules from $G(B)\cup G(C)$ over $G(V_0)$.

Since $\dim(\vec a_2/BC\vec a_1)=|\vec a_2|$, the consistency result follows from Corollary \ref{cor_shift}. Let $\vec a_2'\in G^{|\vec x_2|}$ be such that $\vec a_2' \models p^{\mathcal{L}}_{2BC\vec a_1}(\vec x_2) \cup q_{2BC}^{R}(\vec x_2)$, then $\vec a_2'\models p_{2BC}^{R}$ and we have that $\vec a_2'$ is independent in the sense of $R$-modules from $G(B)\cup G(C)$ over $G(V_0)$. Note that 
$\vec a_1\vec a_2'\equiv^{\mathcal{L}}_{BC}\vec a_1\vec a_2$ and thus $\vec a_1\vec a_2'\ind^0_{V_0}B\cup C$.

Let $\vec a_3'$ be such that 
$\vec a_1\vec a_2'\vec a_3'\equiv^{\mathcal{L}}_{BC}\vec a_1\vec a_2\vec a_3$ and write $\vec a'=\vec a_1\vec a_2' \vec a_3'$ for short.

\textbf{Claim 2.} $\vec a'\models p_B(\vec x)\cup p_C(\vec x)$ and it is $\ind^G$-independent from $B\cup C$ over $V_0$.

Since $\vec a_1 \vec a_2 \vec a_3\ind_{V_0}^0B\cup C$, $\vec a_1\vec a_2'\vec a_3'\equiv^{\mathcal{L}}_{BC}\vec a_1\vec a_2\vec a_3$ we get that part $(1)$ in the definition of $G$-independence holds. By the way we constructed $\vec a_2'$ we know that part $(2)$ in the definition of $\ind^G$-independence also holds.

Let us consider part (3). Let $t\in G(\vect{\vec a'V_0} + \vect{BC})$, then $t=d+\sum_{i}\alpha_i a_i$ for some $d\in \vect{BC}$. We may enumerate
$\vec a_1=(a_1,\dots,a_n)$, $\vec a'_2=(a_{n+1}',\dots,a_{n+k}')$. By construction of $\vec a_1$, we must have $\alpha_i=0$ for $1\leq i\leq n$. Thus $t=d+\sum_{n<i\leq n+k}\alpha_i a_i'$. 

In this part of the argument we will use that $\FF=\hat{R}$. Then we can write $t=d+\sum_{n<i\leq n+k}(p_i/q_i) a_i'$ for some $p_i,q_i\in R$. Then 
$$q_{n+1}\cdots q_{n+k}t=q_{n+1}\cdots q_{n+k}d+\sum_{n<i\leq n+k}p_i q_{n+1}\cdots \hat{q_i} \cdots q_{n+k} a_i'\in G.$$

To simplify the notation, for each $n<i\leq n+k$ write $p_i'=p_i q_{n+1}\cdots \hat{q_i} \cdots q_{n+k}$. Then we get
\[\sum_{n<i\leq n+k}p_i' a_i'\in q_{n+1}\ldots q_{n+k}d+q_{n+1}\ldots q_{n+k}G.\]
By Remark \ref{rk:heir}, there is $v_0\in V_0$ such that $\sum_{n<i\leq n+k}p_i' a_i'\in q_{n+1}\cdots q_{n+k}v_0+q_{n+1}\cdots q_{n+k}G$, so $q_{n+1}\cdots q_{n+k}(v_0-d)\in q_{n+1}\cdots q_{n+k}G$, so $d=v_0+g$ for some $g\in G$. So 
$t=g+v_0+\sum_{n<i\leq n+k}(p_i/q_i) a_i'$, with $g\in G(\langle BC\rangle )$, $v_0+\sum_{n<i\leq n+k}(p_i/q_i) a_i'\in G(\langle \vec a'V_0\rangle )$.

\end{proof}

\begin{cor}
    Assume that $T$ has $SU$-rank one and that $\hat{R}=\FF$. Then $T^G$ is simple and forking independence is given by $\ind^G$.
\end{cor}

\textbf{Question}
Assume that $T$ has $SU$-rank one. Does the independence theorem hold for $\ind^G$
without the assumption $\FF=\hat R$?
In the proof above we only use this assumption for proving part (3) in the amalgam. A positive answer would also 
show that our characterization of forking independence given by $\ind^G$ also works in this more general setting.\\

The following examples give a base theory $T$ which is geometric, $MR(T)=2$, $\ind^0\neq \ind$ and conditions (1),(2) and (3) still capture forking independence. This gives evidence that our characterization of independence works under more general assumptions.

\begin{exam} \label{lovely-one-based} Let $\mathcal{V}=(V,\ldots)$ be a supersimple SU-rank 1 expansion of a vector space over a field $\mathbb{F}$, and we assume that $\acl=\spane_{\mathbb{F}}$ and $T=Th(\mathcal{V})$ has quantifier elimination. Let $T^P$ be the theory of a lovely pair $(\mathcal{V},P)$. It has been shown in \cite{Va} that in this case the algebraic closure in the pair coincides with the algebraic closure in $\mathcal{V}$, i.e. $\acl_P=\acl=\spane_{\mathbb{F}}$ and $T^P$ is supersimple of SU-rank 2. As shown in \cite{Va2}, for any SU-rank 1 theory $T$, the theory of lovely pairs of $T$ has the weak non-finite cover property. In particular, $T^P$ eliminates $\exists^\infty$. Thus, in our setting, $T^P$ is geometric, with algebraic closure given by the linear span. Note also that due to the modularity of the pregeometry, $T^P$ has quantifier elimination (any closed set is $P$-independent). 

Assume that $R=\mathbb{F}$ and consider the dense-codense generic submodule expansion $T^{PG}$ of $T^P$. Thus, $G(V)$ is an $\mathbb{F}$-vector subspace of $V$, and moreover, by the density property, an elementary submodel of $\mathcal{V}$ viewed as an $\mathcal{L}$-structure. The fact that it is an elementary submodel follows (from the Tarski-Vaught test) by density and being closed under $\acl$: if a 1-formula with parameters in $G(V)$ is algebraic, then it is realized in $G(V)$ because $G(V)$ is algebraically closed; if it is non-algebraic, it is realized by density of $G(V)$.

Clearly, $T^{PG}$ is the theory of lovely pairs of $T^P$, where $T^P$ is viewed as a geometric theory. On the other hand, note that $T^{PG}$ is not the theory of lovely pairs of $T^P$ in the sense of \cite{BPV}. In any model $(\mathcal{V}, P, G)$ of $T^{PG}$, for any $a\in \mathcal{V}$ the formula $P(x-a)$ is non-algebraic, and therefore is realized in $G(V)$. It follows that $V=P(V)+G(V)$. To compare this behavior with that of a lovely pair expansion $T^{PQ}$ of $T^P$ in the sense of \cite{BPV}, recall that forking independence in $(\mathcal{V},P)$ (denoted by $\ind^P$) is given by $$A\ind^P_{C}B\iff A\ind_{C\cup P(V)}B,\  \spane_{\mathbb{F}}(AC)\cap\spane_{\mathbb{F}}(BC)=\spane_{\mathbb{F}}(C).$$ The extension property in the sense of a pair $((\mathcal{V},P),Q)$ of models of $T^P$ then implies that for any $v_1,\ldots,v_n\in V$, we can find $w\in V\backslash P(V)$ such that $w\ind^P_{\emptyset} \vec v Q(V)$, and, in particular, $w\ind_{P(V)}\vec v Q(V)$, equivalently, $w\not\in\spane_{\mathbb{F}}(\vec vP(V)Q(V))$. It follows that $V$ has infinite dimension over $P(V)+Q(V)$. Note also that
any pair $(\mathcal{V}, P, G)\models T^{PG}$ fails the extension property (in the sense of $T^P$): the $L_P$-type of $a\not\in P(V)\cup G(V)$ over $\emptyset$ has no non-forking extension  over $G(V)$ realized in $V$.
Indeed, for any such $a$, we have $a=b+c$ where $b\in P(V)$ and $c\in G(V)$. Then $a\nind_{P(V)}G(V)$, and, thus, $a\nind^P_{\emptyset}G(V)$.
\end{exam}

\begin{exam} Let us simplify further the previous example and assume that $T$ is the theory of pure vector spaces over $\FF$ and $(\mathcal{V},P,G)\models T^{PG}$. As explained in the previous example, $V=P(V)+G(V)$ and we can view $(\mathcal{V},P,G)$ as a definable reduct of $(\mathcal{W},P,G)$, a lovely pair of $T^P$ in the sense of \cite{BPV} (so $W$ has infinite dimension over $V$). Note that both $(P(V),P(V)\cap G(V))$ and $(G(V),P(V)\cap G(V))$, viewed as lovely pairs of vector spaces, have full quantifier elimination. Each of them can be viewed as the small model inside a lovely pair of $T^P$ or $T^G$ and no extra structure is induced on them when viewed in $(\mathcal{W},P,G)$, and so, also in $(\mathcal{V},P,G)$. We claim that $MR(\mathcal{V}, P, G)=3$. Indeed,  $MR(P(V),P(V)\cap G(V))=MR(G(V),P(V)\cap G(V))=2$ as each is a lovely pair of pure vector spaces.
 Consider now the definable linear map $m:G\times P\to V$ given by $m(a,b)=a+b$. The map is onto with kernel $K=\{(a,-a):a\in G\cap P$ and we have $MR(K)=1$.  Thus, $MR(V)=4-1=3$.

We now show that independence as described in Definition \ref{forkingVG} agrees with non-forking in $(\mathcal{V}, P, G)$. Notice that condition (1) recovers independence in the pair $(\mathcal{V}, P)$ and independence in the pair strictly extends algebraic independence. Note also that condition (2) corresponds to algebraic independence inside $G$. It remains to show we can witness dependence when an element becomes algebraic over $G$. Let $D,V_0$ as in Definition \ref{forkingVG} and consider a singleton $a$. Suppose we have $a\in \acl(G\cup D)\setminus \acl(G\cup V_0)$, then $a=g+d$ for some $a\in A$ and $g\in G$, so $g=a-d\in G(a+D)$ but
if $a-d\in G(\spane_{\FF}(aV_0))+G(D)$, then
$a=d+g_1+g_2$ with $g_1+g_2\in G(\spane_{\FF}(aV_0))+G(D)$. Then either $g_1=\mu a+v_0$ for $\mu\neq 0$, $v_0\in V_0$ and then $a\in \acl(G\cup V_0)$, a contradiction, or $g_1=v_0$ and thus $a\in \acl(G\cup V_0)$, again a contradiction. The converse
follows from Remark \ref{cosetsforking}.
\end{exam}

\section{Example: Lovely pair of vector spaces with a generic NSOP1 structure on its factor-space}\label{sec:moreexamples}

In this section, we describe a candidate for an example of an NSOP$_1$ not simple theory which is modular pregeometric. This construction shall serve as a general method for constructing modular pregeometric theories of arbitrary combinatorial complexity (stable, simple, NIP, NSOP$_1$). We leave most of the details to the reader. 

Let $\FF$ be any field, and let us denote by $\FF\text{-vs}$ the theory of infinite dimensional $\FF$-vector space in the usual one-sorted language. For $V\models \FF\text{-vs}$ and $\vec a$ a tuple in $V$, we write $\tp^V(\vec a)$ for it type in the vector space language. Similarly, when we add a new sort $S$ with extra structure to $V$, we will write 
$\tp^S(\vec a)$ for the type of the tuple seeing $S$ with the induced structure, provided $\vec a$ belongs to the sort $S$.

\begin{fact}\label{fact_HpriminV}
Let $V\models \FF\mathrm{-vs}$. Let $H\subseteq V$ such that $(V,H)$ is an $H$-structure. Then $H_\mathrm{ind}$, the induced structure on $H$, is trivial. Let $H'\supseteq H_\mathrm{ind}$ be any expansion of $H$ which is NSOP$_1$. Then $(V,H')$ is NSOP$_1$ and for all $\vec a\subseteq V$, we have 
\[\tp^{V}(\vec a)\cup \tp^{H'}(\dcl^{\mathrm{feq}}(H(\spane_\FF(\vec a))))\vdash \tp^{(V,H')}(\vec a)\]
where $\dcl^{\mathrm{feq}}(A)$ is the set of codes for finite subsets of $A$.
\end{fact}

\begin{proof}
Since the theory $\FF\mathrm{-vs}$ is strongly minimal the pair $(V,H)$ is $\omega$-stable, $H$ is stably embedded and there is no outside structure induced on $H_\mathrm{ind}$. Moreover, this last set just models the theory in the empty language with infinitely many elements. The rest follows from the work in \cite[Section 3]{CdEHJRK22}. We check the following properties:
\begin{enumerate}
    \item $\acl^{\eq}(\vec a)\cap H^{\eq} =  \dcl^{\mathrm{feq}}(H(\spane_\FF(\vec a)))$;
    \item $H$ is \textit{algebraically embedded in $(V,H)$}, which means: if $\vec a\ind^{(V,H)}_{\vec c} \vec b$, then \[\acl^{\eq}(\vec a\vec b \vec c)\cap H^{\eq} \subseteq \acl^{\eq}(\acl^{\eq}(\vec a\vec c)\cap H^{\eq}, \acl^{\eq}(\vec b\vec c)\cap H^{\eq})\]
    (see \cite[Definition 4.9]{CdEHJRK22}).
\end{enumerate} 
(1) follows from the fact that $H_\mathrm{ind}$ and $V$ have weak elimination of imaginaries, and the description of algebraic closure in $H$-structures (see Corollary 3.14 in \cite{BeVa-H} 
\end{proof}

\begin{fact}\label{fact_Vproj} Let $(V,W)$ be a lovely pair of models of the theory $\FF\text{-vs}$. Consider the two sorted structure $(V,V/W,\pi)$ with the quotient homomorphism $\pi:V\rightarrow V/W$ and let $V^* = V/W$ the quotient vector space.
\begin{enumerate}
    \item The theory of $(V,V^*,\pi)$ is stable. 
    \item For all $a\subseteq V$, $\alpha\subseteq V^*$ we have $tp^{V^*}(\alpha/\pi(a))\vdash tp^{(V,V^*)}(\alpha/a)$.
    \item The induced structure on $V^*$ in $(V,V^*, \pi)$ is the one of a pure vector space. In particular $V^*$ is stably embedded and $V^*_\mathrm{ind} = V^*$ has  weak elimination of imaginaries.
    \item $(V,V^*,\pi)$ has weak elimination of imaginaries.
    \item For all $\vec a\subseteq V$ and $\vec \alpha\subseteq V^*$, $\acl^{\eq} (\vec a,\vec \alpha) = \dcl^{\mathrm{feq}}(\spane_\FF(\vec a)\cup \spane_\FF(\pi(\vec a),\vec \alpha))$, in particular $\acl^{\eq} (\vec a,\vec \alpha)\cap {V^*}^{\eq} = \dcl^{\mathrm{feq}}(\spane_\FF(\pi(\vec a),\vec \alpha))$.
  
\end{enumerate}
\end{fact}
\begin{proof}
This is left to the reader as an exercise, it is very similar to the work in \cite[Section 3]{DeACFG}.
\end{proof}

\begin{cor}
Let $(V, (V^*,H'))$ be the structure obtained by expanding the structure on $V^*$ by an $H$-predicate and an NSOP$_1$ expansion $H'\supseteq H$. Then the expansion $(V,(V^*,H'))$ is NSOP$_1$. Furthermore, for each $\vec a \vec \alpha$ we have 
\[\tp^{(V,V^*)}(\vec a \vec \alpha )\cup \tp^{H'}(\dcl^{\mathrm{feq}}(H(\spane_\FF(\pi(\vec a),\vec \alpha))))\vdash \tp^{(V,(V^*,H'))}(\vec a\vec \alpha).\]
\end{cor}
\begin{proof}
It will again follow from \cite[Section 3]{CdEHJRK22}. By Fact \ref{fact_HpriminV}, $(V^*,H')$ is an NSOP$_1$ expansion of the stably embedded set $V^*$ inside $(V,V^*)$. To preserve NSOP$_1$, one has to check that $V^*$ is algebraically embedded in $(V,V^*)$, which is the following: for $\ind^f$ forking independence in $(V,V^*)$, if  $\vec a\vec \alpha\ind^f _{\vec c\vec \gamma} \vec b\vec \beta$, and $\vec c\vec \gamma\subseteq \vec a\vec \alpha\cap \vec b\vec \beta$, then
\[\acl^{\eq} (\vec a\vec b\vec \alpha\vec \beta)\cap {V^*}^{\eq} \subseteq  \acl^{\eq} (\acl^{\eq} (\vec a\vec \alpha)\cap {V^*}^{\eq} , \acl^{\eq} (\vec b\vec \beta)\cap {V^*}^{\eq})\]
This easily follows by Fact \ref{fact_Vproj} (5), (regardless of the forking condition): \[\spane_\FF(\pi(\vec a\vec b)\vec \alpha\vec \beta)= \spane_\FF(\spane_\FF(\pi(\vec a)\vec \alpha),\spane_\FF(\pi(\vec b)\vec \beta)),\]
so
\[\dcl^{\mathrm{feq}}(\spane_\FF(\pi(\vec a\vec b)\vec \alpha\vec \beta))= \dcl^{\mathrm{feq}}(\spane_\FF(\dcl^{\mathrm{feq}}(\spane_\FF(\pi(\vec a)\vec \alpha)),\dcl^{\mathrm{feq}}(\spane_\FF(\pi(\vec b)\vec \beta)))),\]
which proves that $V^*$ is algebraically embedded in $(V,V^*)$.
\end{proof}

Now, we consider the structure $V^\pi$ whose universe is $V$ with the structure $(V^*, H')$ pulled back under $\pi$. By this we mean that there is a predicate $H^\pi\subseteq V$ such that 
\[x\in H^\pi \iff \pi(x)\in H.\]
Also, for any relation $R\subseteq H^n$ in $H'$, define $R^\pi\subseteq (H^\pi)^n$ so that
\[x\in R^\pi \iff \pi(x)\in R.\]
For function symbol $f$ and constant symbol $c$, define the relations $R_f^\pi$ and $R_c^\pi$ as follows:
\[x\in R_f^\pi \iff \pi(x)\in \mathrm{Graph}(f).\]
\[x\in R_c^\pi \iff \pi(x) = c.\]
We also keep the predicate $W$ in the structure $V^\pi$.
\begin{cor}
The structure $V^\pi$ is NSOP$_1$. If $H'$ is not simple, neither is $V^\pi$.
\end{cor}
\begin{proof}
Recall that $(V,W)$ is a lovely pair and $\pi: V\to V/W$ is the projection and $W$ is kept as a predicate inside $V^{\pi}$. The structure $V^\pi$ is interdefinable with $(V,(V^*, H'))$ so $V^\pi$ is NSOP$_1$. Also if $H'$ is not simple, $V^\pi$ is not simple.
\end{proof}

Now even if the algebraic closure in $(V^*,H')$ might be larger than the vector space span, the algebraic closure should not change in $V^\pi$ because $\pi$ has infinite fibers. Every formula $\phi(x)$ in $(V^*, H')$ has a dual formula $\phi^\pi(x)$ in $V^\pi$ such that 
$V^\pi\models \phi^\pi (a)\iff V^*\models \phi(\pi(a))$
so no algebraicity should come from the extra structure in $V^\pi$. For the same reason, uniform finiteness should be preserved. This leads us to state as a conjecture:\\

\textit{Conjecture.} In $V^\pi$, the algebraic closure of a set $A\subset V$ is given by the $\spane_{\FF}(A)$, and $V^\pi$ eliminates $\exists^\infty$. In particular the theory of $V^\pi$ is geometric. \\

\end{document}